\documentclass[reqno]{amsart}
\usepackage{amsfonts,amsmath,amsthm,amssymb,mathrsfs}
\usepackage{color}
\usepackage{amsmath,amssymb,amsfonts,bm}
\usepackage{graphicx}
\usepackage{CJKutf8}
\usepackage{newunicodechar}
\usepackage{xcolor}
\numberwithin{equation}{section}
\usepackage[pdfstartview=FitH,colorlinks=true]{hyperref}
\hypersetup{urlcolor=red, linkcolor=blue,citecolor=red, CJKbookmarks=true}
\allowdisplaybreaks

\theoremstyle{plain}
\newtheorem{theorem}{Theorem}[section]
\newtheorem{remark}[theorem]{Remark}
\newtheorem{lemma}[theorem]{Lemma}
\newtheorem{proposition}[theorem]{Proposition}
\newtheorem{definition}[theorem]{Definition}
\newtheorem{corollary}[theorem]{Corollary}

\numberwithin{equation}{section}

\begin{document}

\title[Non-cutoff Boltzmann equation]
{The Cauchy problem for radially symmetric homogeneous Boltzmann
equation\\
with Shubin class initial datum\\
and Gelfand-Shilov smoothing effect}

\author[Hao-Guang Li \& Chao-Jiang Xu]
{Hao-Guang Li and Chao-Jiang Xu}

\address{Hao-Guang Li,
\newline\indent
School of Mathematics and Statistics, South-Central University for Nationalities
\newline\indent
430074, Wuhan, P. R. China}
\email{lihaoguang@mail.scuec.edu.cn}

\address{Chao-Jiang Xu,
\newline\indent
Universit\'e de Rouen, CNRS UMR 6085, Laboratoire de Math\'ematiques
\newline\indent
76801 Saint-Etienne du Rouvray, France
\newline\indent
School of Mathematics and statistics, Wuhan University 430072,
Wuhan, P. R. China
}
\email{chao-jiang.xu@univ-rouen.fr}

\date{\today}



\begin{abstract}
In this paper, we study the Cauchy problem for radially symmetric homogeneous non-cutoff Boltzmann equation with Maxwellian molecules, the initial datum belongs to Shubin space of the negative index which can be characterized by spectral decomposition of the harmonic oscillators. The Shubin space of the negative index contains the measure functions.  Based on this spectral decomposition, we construct the weak solution with Shubin class initial datum, we also prove that the Cauchy problem enjoys Gelfand-Shilov smoothing effect, meaning that the smoothing properties are the same as the Cauchy problem defined by the evolution equation associated to a fractional harmonic oscillator.
\end{abstract}

\maketitle
\tableofcontents


\section{Introduction}\label{S1}
In this work, we consider the spatially homogeneous Boltzmann equation
\begin{equation}\label{eq1.10}
\left\{
\begin{array}{ll}
   \partial_t f= Q(f,f),\\
  f|_{t=0}=f_0\geq0,
\end{array}
\right.
\end{equation}
where $f=f(t,v)$ is the density distribution function depending on the variables
$v\in\mathbb{R}^{3}$ and the time $t\geq0$. The Boltzmann bilinear collision operator is given by
\begin{equation*}
Q(g,f)(v)
=\int_{\mathbb{R}^{3}}
\int_{\mathbb{S}^{2}}
B(v-v_{\ast},\sigma)
(g(v_{\ast}^{\prime})f(v^{\prime})-g(v_{\ast})f(v))
dv_{\ast}d\sigma,
\end{equation*}
where for $\sigma\in \mathbb{S}^{2}$,~the symbols~$v_{\ast}^{\prime}$~and~$v^{\prime}$~are abbreviations for the expressions,
$$
v^{\prime}=\frac{v+v_{\ast}}{2}+\frac{|v-v_{\ast}|}{2}\sigma,\,\,\,\,\, v_{\ast}^{\prime}
=\frac{v+v_{\ast}}{2}-\frac{|v-v_{\ast}|}{2}\sigma,
$$
which are obtained in such a way that collision preserves momentum and kinetic energy,~namely
$$
v_{\ast}^{\prime}+v^{\prime}=v+v_{\ast},\,\,\,\,\, |v_{\ast}^{\prime}|^{2}+|v^{\prime}|^{2}=|v|^{2}+|v_{\ast}|^{2}.
$$
The non-negative cross section $B(z,\sigma)$~depends only on $|z|$~and the scalar product $\frac{z}{|z|}\cdot\sigma$.~For physical models,~it usually takes the form
$$
B(v-v_{\ast},\sigma)=\Phi(|v-v_{\ast}|)b(\cos\theta),~~~~
\cos\theta=\frac{v-v_{\ast}}{|v-v_{\ast}|}\cdot\sigma,~~ 0\leq\theta\leq\frac{\pi}{2}.
$$
Throughout this paper, we consider the Maxwellian molecules case which corresponds to the case $\Phi\equiv 1$ and focus our attention on the following general assumption of $b$
\begin{equation}\label{beta}
\beta(\theta)=
2\pi b(\cos2\theta) \sin2\theta \approx \theta^{-1-2s},\,\,\mbox{when}~\theta\rightarrow0^{+},
\end{equation}
for some $0<s<1$. Without loss of generality, we may assume that $b(\cos\theta)$ is supported on the set $\cos\theta\geq0$.
See for instance \cite{NYKC1} for more details on $\beta(\,\cdot\,)$ and \cite{Villani} for a general collision kernel.

We introduce the fluctuation of density distribution function
$$
f(t,v)=\mu(v)+\sqrt{\mu}(v)g(t,v)
$$
near the absolute Maxwellian distribution
$$
\mu(v)=(2\pi)^{-\frac 32}e^{-\frac{|v|^{2}}{2}}.
$$
Then the Cauchy problem \eqref{eq1.10} is reduced to the Cauchy problem for the fluctuation
\begin{equation} \label{eq-1}
\left\{ \begin{aligned}
         &\partial_t g+\mathcal{L}(g)={\bf \Gamma}(g, g),\quad t>0, v\in\mathbb{R}^3,\\
         &g|_{t=0}=g^{0},\quad v\in\mathbb{R}^3.
\end{aligned} \right.
\end{equation}
with $g^0(v)=\mu^{-\frac 12}f_0(v) -\sqrt{\mu}(v)$, where
$$
{\bf \Gamma}(g, g)=\mu^{-\frac 12}Q(\sqrt{\mu}g,\sqrt{\mu}g),\quad
\mathcal{L}(g)=-\mu^{-\frac 12}\Big(Q(\sqrt{\mu}g,\mu)+Q(\mu,\sqrt{\mu}g)\Big).
$$
The linear operator $\mathcal{L}$ is nonnegative (\cite{NYKC1,NYKC2,NYKC3})\,with the null space
$$
\mathcal{N}=\text{span}\left\{\sqrt{\mu},\,\sqrt{\mu}v_1,\,\sqrt{\mu}v_2,\,
\sqrt{\mu}v_3,\,\sqrt{\mu}|v|^2\right\}.
$$
It is well known that the angular singularity in the cross section leads to the regularity of the solution, see\cite{MU,YSCT,Villani,TZ} and the references therein.  We can also refer to \cite{G-N,L-X} for smoothing effect of the radially symmetric spatially
homogeneous Boltzmann equation.
Regarding the linearized Cauchy problem \eqref{eq-1}, the global in time smoothing effect of the solution to the Cauchy problem \eqref{eq-1} has been shown, in \cite{NYKC3} for radially symmetric case ,and in \cite{GLX} for general case with the initial data in $L^2(\mathbb{R}^3)$.   It proved that the solutions of the Cauchy problem \eqref{eq-1}
belong to the symmetric Gelfand-Shilov  space $S^{\frac{1}{2s}}_{\frac{1}{2s}}(\mathbb{R}^3)$ for any positive time.  Moreover, there exist positive constants $c>0$ and $C>0$, such that
$$
 \forall\, t>0,\quad \|e^{ct \mathcal{H}^s}g(t)\|_{L^2}\leq\,C\|g_0\|_{L^2},
$$
where $\mathcal{H}$ is the harmonic oscilator
\begin{equation*}
  \mathcal{H}=-\triangle_v +\frac{|v|^2}{4}.
\end{equation*}
The Gelfand-Shilov space $S^{\mu}_{\nu}(\mathbb{R}^3)$, with $\mu,\,\nu>0,$\,$\mu+\nu\geq1,$\, is the subspace of smooth functions satisfying:
$$
\exists\, A>0,\, C>0,\,
\sup_{v\in\mathbb{R}^3}|v^{\beta}\partial^{\alpha}_vf(v)|\leq\,CA^{|\alpha|+
|\beta|}(\alpha!)^{\mu}(\beta!)^{\nu},\,\,\forall\,\alpha,\,\beta\in\mathbb{N}^3.
$$
So that Gelfand-Shilov class $S^{\mu}_{\nu}(\mathbb{R}^3)$ is Gevery class  $G^{\mu}(\mathbb{R}^3)$ with rapid decay at infinite. This Gelfand-Shilov space can be characterized as the subspace of Schwartz functions $f\in\,\mathscr{S}(\mathbb{R}^3)$ such that,
$$
\exists\, C>0,\,\epsilon>0,\,|f(v)|\leq Ce^{-\epsilon|v|^{\frac{1}{\nu}}},\,\,v\in\mathbb{R}^3\,\,\text{and}\,\,|\hat{f}(\xi)|\leq Ce^{-\epsilon|\xi|^{\frac{1}{\mu}}},\,\,\xi\in\mathbb{R}^3.
$$
The symmetric Gelfand-Shilov space $S^{\nu}_{\nu}(\mathbb{R}^3)$ with $\nu\geq\frac{1}{2}$ can  also be identified with
$$
S^{\nu}_{\nu}(\mathbb{R}^{3})=\left\{f\in C^\infty (\mathbb{R}^3);  \exists \tau>0,
\|e^{\tau \, \mathcal{H}^{\frac{1}{2\nu}}}f\|_{L^2}<+\infty\right\}.
$$
See Appendix \ref{appendix} for more properties of Gelfand-Shilov spaces.

In this paper, we will show that the solutions to the Cauchy problem \eqref{eq-1} belong to the Gelfand-Shlov space $S^{\frac{1}{2s}}_{\frac{1}{2s}}(\mathbb{R}^3)$ for any positive time, but with the initial datum in  the Shubin space $Q^{-\frac{3}{2}-\alpha}(\mathbb{R}^3)$ for $0<\alpha<2s$.
This space contains the Sobolev space $H^{-\frac{3}{2}-\alpha}(\mathbb{R}^3)$,
so it is  more singular than the measure valued initial datum. For $\beta\in \mathbb{R}$,
Shubin  introduce the following function spaces, (see \cite{Shubin}, Ch. IV, 25.3)
\[
Q^{\beta}(\mathbb{R}^3)=\Big\{u\in \mathcal{S}'(\mathbb{R}^3);\,\,   \|u\|_{Q^{\beta}(\mathbb{R}^3)} =\bigl\|\bigl(\mathcal{H} + 1 \bigr)^{\frac{\beta}{2}} \, u\bigr\|_{L^2(\mathbb{R}^3)}
<+\infty\Big\}.
\]
We denote  by $Q^{\beta}_r(\mathbb{R}^3)$  the radial symmetric functions belongs to  $Q^{\beta}(\mathbb{R}^3)$.

The main theorem of this paper is given in the following.
\begin{theorem}\label{trick}
For any $\alpha$ satisfying $0<\alpha<2s$ with $s$ given in $\eqref{beta}$,
there exists $\varepsilon_0>0$ such that for any initial datum
$g^0\in\, Q^{-\frac{3}{2}-\alpha}_r(\mathbb{R}^3)\cap \mathcal{N}^{\perp}$
with $\|g^0\|^2_{Q^{-\frac{3}{2}-\alpha}(\mathbb{R}^3)}\le \varepsilon_0$,
the Cauchy problem \eqref{eq-1} admits a global radial symmetric weak solution
$$
g\in L^{+\infty}([0, +\infty[; \, Q^{-\frac{3}{2}-\alpha}_r(\mathbb{R}^3) ).
$$
Moreover, there exists  $c_0>0$, $C>0$, such that, for any $t\ge 0$,
$$
\|e^{{c_0t} \, \mathcal{H}^s}\mathcal{H}^{-\frac{3}{4}-\frac{\alpha}{2}}g(t)\|_{L^2(\mathbb{R}^3)}
\leq\,
Ce^{-\frac{\lambda_{2}}{4} t} \, \|g_0\|_{Q^{-\frac{3}{2}-\alpha}(\mathbb{R}^3)},
$$
where
$$
\lambda_{2}=\int^{\frac{\pi}{4}}_{-\frac{\pi}{4}}\beta( \theta)(1-\sin^4\theta-\cos^4\theta)d\theta>0.
$$
This implies that $ g(t)\in S^{\frac{1}{2s}}_{\frac{1}{2s}}(\mathbb{R}^3)$ for any $t>0$.
\end{theorem}
\begin{remark}
It is well known that  the single Dirac mass is the constant solution of the Cauchy problem \eqref{eq1.10}. The following example is some how surprise. Let
$$
f_0=\delta_{\mathbf{0}}-\left(\frac{3}{2}-\frac{|v|^2}{2}\right)\mu
$$
be the initial datum of Cauchy problem \eqref{eq1.10}, then
$f_0=\mu+\sqrt{\mu}g^0$ with
\begin{equation}\label{ex1}
g^0=\frac{1}{\sqrt{\mu}}\delta_{\mathbf{0}}-\sqrt{\mu}-
\left(\frac{3}{2}-\frac{|v|^2}{2}\right)\sqrt{\mu}.
\end{equation}
We will prove, in the Section \ref{S2},  that $g^0\in Q^{-\frac{3}{2}-\alpha}_r(\mathbb{R}^3)\cap \mathcal{N}^{\perp}$, so Theorem \ref{trick} imply that the Cauchy problem \eqref{eq1.10} admit a global solution
$$
f=\mu+\sqrt{\mu}g\in  L^{+\infty}([0, +\infty[; \, Q^{-\frac{3}{2}-\alpha}_r(\mathbb{R}^3) ) \cap
C^{0}(]0, +\infty[; \, S^{\frac{1}{2s}}_{\frac{1}{2s}}(\mathbb{R}^3) ).
$$
\end{remark}

\begin{remark} It is also surprise that the nonlinear operators is well-defined for 
the Shubin space, in fact, for any $0<\alpha<2s<2$,  and $f, g \in Q^{-\frac{3}{2}-\alpha}(\mathbb{R}^3)$, we prove 
\begin{equation}\label{tri}
\Gamma(f,\, g)\in Q^{-2s-\frac{3}{2}-\alpha-\gamma}(\mathbb{R}^3)	
\end{equation}
for any $\gamma>1$ (see Corollary \ref{cor1}).
\end{remark}

In the study of the Cauchy problem of the homogeneous Boltzmann equation in Maxwellian molecules case,
Tanaka in \cite{Tanaka} proved the existence and the uniqueness of the solution under the assumption of the initial data  $f_0>0$,
\begin{equation}\label{finite-energy}
\int_{\mathbb{R}^3}f_0(v)dv=1, \quad \int_{\mathbb{R}^3}v_jf_0(v)dv=0,\,j=1,2,3, \quad\int_{\mathbb{R}^3}|v|^2f_0(v)dv=3.
\end{equation}
The proof of this result was simplified and generalized in \cite{Villani1998} and \cite{TV}.
Cannone-Karch in \cite{Cannone-Karch} extended this result for the initial data of the probability measure without \eqref{finite-energy}, this means that the initial data could have infinite energy.  Morimoto \cite{Morimoto2} and Morimoto-Yang \cite{MY} extended this result more profoundly and prove the smoothing effect of the solution to the Cauchy problem \eqref{eq1.10} with the measure initial data which is not contain in $P_{\alpha}(\mathbb{R}^3)$, where $P_{\alpha}(\mathbb{R}^3)$, $0\leq\alpha\leq2$ is the probability measures $F$ on $\mathbb{R}^3$ such that
$$
\int_{\mathbb{R}^3}|v|^{\alpha}dF(v)<\infty,
$$
and moreover when $\alpha\geq1,$ it requires that
$$
\int_{\mathbb{R}^3}v_jdF(v)=0,\,\,j=1,2,3.
$$
Recently, Morimoto, Wang and Yang in \cite{Morimoto-Wang-Yang} introduce a new classification on the characteristic functions  and prove the smoothing effect under this measure initial datum, see also \cite{Morimoto-Wang-Yang2} and \cite{Cho-Morimoto-Wang-Yang}. For the Shubin space, we have

\begin{remark}
Let $0<\beta<\alpha<2$, then for any $g^0\in \,Q^{-\frac{3}{2}-\beta}_r(\mathbb{R}^3)$  with $\langle g^0, \sqrt{\mu}\rangle=0$,  we have
\begin{equation}\label{ex2}
f_0=\mu+\sqrt{\mu}\, g^0 \in  P_{\alpha}(\mathbb{R}^3),
\end{equation}
So that for a class of measure initial datum, we prove the Gelfand-Shilov smoothing effect of the Cauchy problem \eqref{eq1.10}.
\end{remark}

The rest of the paper is arranged as follows:
In Section \ref{S2}, we introduce the spectral analysis
of the linear and nonlinear Boltzmann operators,
and present the explicit solution of the Cauchy problem \eqref{eq-1} by
transforming the linearized Boltzmann equation into an infinite system of ordinary differential
equations which can be solved explicitly.  Furthermore, we prove \eqref{ex2} and interpret the Example \eqref{ex1}.
In Section \ref{S3}, we establish an upper bounded estimates
of the nonlinear operators with an exponential weighted norm.
The proof of the main Theorem \ref{trick} will be presented in Section \ref{S4}-\ref{S5}.
In the Appendix \ref{appendix}, we present some identity properties
of the Gelfand-Shilov spaces and the Shubin spaces used in this paper.

\section{Preliminary }\label{S2}

\noindent{\bf Diagonalization of  the linear operators.} We first recall the spectral decomposition of {\color{black} the} linear Boltzmann operator.
Let
$$
\varphi_{n}(v)=\sqrt{\frac{n!}{4\sqrt{2}\pi\Gamma(n+\frac{3}{2})}}L^{(\frac{1}{2})}_n\left(\frac{|v|^{2}}{2}\right)e^{-\frac{|v|^2}{4}},
$$
where $\Gamma(\,\cdot\,)$ is the standard Gamma function, for any $x>0$,
$$
\Gamma(x)=\int^{+\infty}_0t^{x-1}e^{-x}dx,
$$
and the Laguerre polynomial $L^{(\alpha)}_{n}$~of order $\alpha$,~degree $n$ read,
$$
L^{(\alpha)}_{n}(x)=\sum^{n}_{r=0}(-1)^{n-r}\frac{\Gamma(\alpha+n+1)}{r!(n-r)!
\Gamma(\alpha+n-r+1)}x^{n-r}.
$$
Then  $\left\{\varphi_{n}\right\}$ constitute an orthonormal basis of $L^2_{rad}(\mathbb{R}^3)$, the radially symmetric function space (see \cite{NYKC3}). In particular,
\begin{align*}
\varphi_{0}(v)&=(2\pi)^{-\frac{3}{4}}e^{-\frac{|v|^2}{4}}=\sqrt{\mu},\\
\varphi_{1}(v)&=\sqrt{\frac{2}{3}}\left(\frac{3}{2}-\frac{|v|^2}{2}\right)(2\pi)^{-\frac{3}{4}}e^{-\frac{|v|^2}{4}}
=\sqrt{\frac{2}{3}}\left(\frac{3}{2}-\frac{|v|^2}{2}\right)\sqrt{\mu}.
\end{align*}
Furthermore, we have, for suitable radial symmetric function $g$,
\begin{equation*}
\mathcal{H}(g)=\sum^{\infty}_{n=0}(2n+\frac 32)\, g_n\, \varphi_{n}\, ,
\end{equation*}
where $g_n=\langle g,\, \varphi_{n}\rangle$ and
\begin{equation*}
\mathcal{L}(g)=\sum^{\infty}_{n=0}
\lambda_{n}\, g_n\, \varphi_{n}\,
\end{equation*}
with $\lambda_{0}=\lambda_{1}=0$ and for $n\geq2$,
$$
\lambda_{n}=\int^{\frac{\pi}{4}}_{-\frac{\pi}{4}}\beta(\theta)
(1-(\cos\theta)^{2n}-(\sin\theta)^{2n})d\theta >0 .
$$
Using this spectral decomposition, the definition of $\mathcal{H}^\alpha,\, e^{c\mathcal{H}^s},\, e^{c\mathcal{L}}$ is then classical.
More explicitly, we have
\begin{align*}
&\|\mathcal{L}^{\frac{1}{2}}g\|^2_{L^2}=\sum^{\infty}_{n=2}
\lambda_{n}|g_n|^2; \\
&\|\mathcal{H}^{-\frac{\alpha}{2}}g\|^2_{L^2}=\sum^{\infty}_{n=0}
(2n+\frac{3}{2})^{-\alpha}|g_n|^2.
\end{align*}

\smallskip
\noindent{\bf Triangular effect of the non linear operators.} We study now the algebra property of the nonlinear terms
$$
{\bf \Gamma}(\varphi_{n},
\varphi_{m}),
$$
By the same proof of Proposition 2.1 in \cite{GLX} or Lemma 3.3 in \cite{NYKC3}, we have the following triangular effect for the nonlinear Boltzmann operators on the basis $\{\varphi_{n}\}$.

\begin{proposition}\label{expansion}
The following algebraic identities hold,
\begin{align*}
&(i_1) \quad\,\, {\bf \Gamma}(\varphi_{0},\varphi_{m})=
\left(\int^{\frac{\pi}{4}}_0\beta(\theta)
((\cos\theta)^{2m}-1)d\theta\right)\varphi_{m},\,\, m\in\mathbb{N};\\
&(i_2) \quad\,\, {\bf \Gamma}(\varphi_{n},\varphi_{0})=\left(\int^{\frac{\pi}{4}}_0\beta(\theta)((\sin\theta)^{2n}-\delta_{0,n})d\theta\right)\varphi_{n},\,\,n\in\mathbb{N};\\
&(ii)  \quad\,\, {\bf \Gamma}(\varphi_{n},\varphi_{m})=\mu_{n,m}\varphi_{n+m}, \,\,\text{for}\,\, n\geq1, m\geq1,
\end{align*}
where
\begin{align}\label{mumn}
\mu_{n,m}=\sqrt{\frac{(2n+2m+1)!}{(2n+1)!(2m+1)!}}
\left(\int^{\frac{\pi}{4}}_0\beta(\theta)(\sin\theta)^{2n}(\cos\theta)^{2m}d\theta\right).
\end{align}
\end{proposition}

\begin{remark}\label{remark-decomp}
Obviously, we can deduce from $(i_1)$ and $(i_2)$ of Proposition \ref{expansion} that
\begin{equation*}
\quad\forall n\in\mathbb{N}, \quad{\bf \Gamma}(\varphi_{0},\varphi_{n})+{\bf \Gamma}(\varphi_{n},\varphi_{0})=-\lambda_{n}\, \varphi_{n}.
\end{equation*}
It is well known that the linearized radially symmetric Boltzmann operator $\mathcal{L}$ behave as a fractional harmonic oscillator $\mathcal{H}^s$ with $s$ given in \eqref{beta}, see \cite{NYKC1}.
Moreover, from Theorem 2.2 in \cite{NYKC2},
 for $n\in\mathbb{N}$ and $n\geq2$,
\begin{equation}\label{eq:3.111}
 \lambda_{n}\approx n^s.
\end{equation}
Namely, there exists a constant $c_0>0$, such that,
  $$c_0 n^s\leq\lambda_{n}\leq\frac{1}{c_0} n^s.$$
We can also refer to \cite{HAOLI}.
\end{remark}

\noindent{\bf Explicit solution of the Cauchy problem \eqref{eq-1}.}
 Now we solve explicitly the Cauchy problem associated to the non-cutoff radial symmetric spatially
homogeneous Boltzmann equation with Maxwellian molecules for a small $Q^{-\frac{3}{2}-\alpha}_r(\mathbb{R}^3)$-initial radial data.

We search a radial solution to the Cauchy problem \eqref{eq-1} in the form
$$
g(t)=\sum^{+\infty}_{n=0}g_{n}(t)\varphi_{n}\, \,\,\mbox{with}\,\,\, g_n(t)=\left\langle g(t),\, \varphi_{n}\right\rangle
$$
with initial data
$$
g|_{t=0}=g^0=\sum^{+\infty}_{n=0}\left\langle g^0,\, \varphi_{n}\right\rangle\varphi_{n}\, .
$$
Remark that $g^0\in Q^{-\frac{3}{2}-\alpha}_r(\mathbb{R}^3)$ is equivalent to $g^0$ radial and
$$
\|g_0\|^2_{Q^{-\frac{3}{2}-\alpha}_r(\mathbb{R}^3)}=\sum_n n^{-\frac{3}{2}-\alpha}\left|\langle g^0,\, \varphi_{n}\rangle\right|^2<+\infty,
$$
see appendix \ref{appendix}.

It follows from Proposition \ref{expansion} and Remark \ref{remark-decomp} that, for convenable radial symmetric function $g$, we have
\begin{align*}
{\bf \Gamma}(g,g)
&=-\sum^{+\infty}_{n=0}
g_0(t)g_n(t)
\lambda_{n}\varphi_{n}\\
&\quad+\sum^{+\infty}_{n=1}\sum^{+\infty}_{m=1}g_n(t)g_m(t)\mu_{n,m}
\varphi_{m+n},
\end{align*}
where $\mu_{n,m}$ was defined in \eqref{mumn}.  This implies that,
\begin{align*}
{\bf \Gamma}(g,g)
&=\sum^{+\infty}_{n=0}
\Big[
-g_0(t)g_n(t)\lambda_{n}
+\sum_{\substack{k+l=n\\k\geq1,l\geq1}}g_k(t)g_l(t)\mu_{k,l}
\Big]
\varphi_{n}.
\end{align*}
For radial symmetric function $g$, we also have
\begin{equation*}
\mathcal{L}( g )=\sum^{+\infty}_{n=0}\lambda_{n}\,g_n(t)\,\varphi_{n}.
\end{equation*}
Formally, we take inner product with $\varphi_{n}$ on both sides of \eqref{eq-1},
we find that the functions $\{g_{n}(t)\}$ satisfy the following infinite system
of the differential equations
\begin{equation}\label{ODE-1}
\partial_t g_n(t)+\lambda_{n}\,g_n(t)=-g_0(t)g_n(t)\lambda_{n}
+\sum_{\substack{k+l=n\\k\geq1,l\geq1}}g_k(t)g_l(t)\mu_{k,l},\quad\forall n\in\mathbb{N},
\end{equation}
with initial data
$$g_n(0)=\left\langle g^0,\varphi_{n}\right\rangle, \quad\forall n\in\mathbb{N}.$$
Consider that $g^0\in Q^{-\frac{3}{2}-\alpha}_r(\mathbb{R}^3)\cap \mathcal{N}^{\perp} $ and $\lambda_{0}=\lambda_{1}=0$, we have
$$g_0(0)=g_1(0)=0.$$
The infinite system of the differential equations \eqref{ODE-1} reduces to be
\begin{equation}\label{ODE-2}
\left\{ \begin{aligned}
         &g_0(t)=g_1(t)=0\\
         &\partial_t g_n(t)+\lambda_{n}\,g_n(t)=\sum_{\substack{k+l=n\\k\geq2,l\geq2}}\mu_{k,l}g_k(t)g_l(t),\forall n\geq2, \\
         &g_n(0)=\left\langle g^0,\varphi_{n}\right\rangle.
\end{aligned} \right.
\end{equation}
On the right hand side of the second equation in \eqref{ODE-2}, the indices $k$ and $l$ are always strictly less than $n$, then this system of the differential equations is triangular, which can be explicitly solved while solving a sequence of linear differential equations.

The proof of Theorem \ref{trick} is reduced to prove the convergence of following series
\begin{equation}\label{ODE}
g(t)=\sum^{+\infty}_{n=2}g_n(t)\varphi_{n}
\end{equation}
in the function space $Q^{-\frac{3}{2}-\alpha}_r(\mathbb{R}^3)$.

\smallskip
\noindent{\bf Measures spaces.} Following Cannone and Karch \cite{Cannone-Karch}, the Fourier transform of a probability measure $F\in P_0(\mathbb{R}^3)$ called a characteristic function which is
$$
\varphi(\xi)=\hat{f}(\xi)=\mathcal{F}(F)(\xi)=\int_{\mathbb{R}^3}e^{-iv\cdot\xi}dF(v).
$$
Note that $\int_{\mathbb{R}^3}dF(v)=1$.  Set $\mathcal{K}= \mathcal{F}(P_0(\mathbb{R}^3))$, inspired by \cite{TV} and the references, Cannone and Karch \cite{Cannone-Karch} defined a subspace $\mathcal{K}^{\alpha}$ for $\alpha\geq 0$ as follows:
\begin{equation*}
\mathcal{K}^{\alpha}=\{\varphi\in\mathcal{K}; \|\varphi-1\|_{\alpha}<+\infty\},
\end{equation*}
where
\begin{equation*}
\|\varphi-1\|_{\alpha}=\sup_{\xi\in\mathbb{R}^3}\frac{|\varphi(\xi)-1|}{|\xi|^{\alpha}}.
\end{equation*}
The space $\mathcal{K}^{\alpha}$ endowed with the distance
\begin{equation*}
\|\varphi-\psi\|_{\alpha}=\sup_{\xi\in\mathbb{R}^3}\frac{|\varphi(\xi)-\psi(\xi)|}{|\xi|^{\alpha}}
\end{equation*}
is a complete metric space (see Proposition 3.10 in \cite{Cannone-Karch}).
However, in this classification, the space $\mathcal{K}^{\alpha}$ is strictly bigger than $\mathcal{F}(P_{\alpha}(\mathbb{R}^3))$ for $\alpha\in]0, 2[.$
Indeed, it is shown in Remark 3.16 of \cite{Cannone-Karch} that the function $\varphi_{\alpha}(\xi)=e^{-|\xi|^{\alpha}}$ with $\alpha\in ]0, 2[$ ,
belongs to $\mathcal{K}^{\alpha}$ but $\mathcal{F}^{-1}(\varphi_{\alpha})$ is not contained in $P_{\alpha}(\mathbb{R}^3)$. In order to fill this gap, Morimoto, Wang and Yang in \cite{Morimoto-Wang-Yang} introduce a classification on the characteristic functions
$$
\mathcal{M}^{\alpha}=\{\varphi\in\mathcal{K}^{\alpha};\|\varphi-1\|_{\mathcal{M}^{\alpha}}<+\infty\}, \,\alpha\in ]0, 2[\, ,
$$
where
$$
\|\varphi-1\|_{\mathcal{M}^{\alpha}}=\int_{\mathbb{R}^3}\frac{|\varphi(\xi)-1|}{|\xi|^{3+\alpha}}d\xi\, .
$$
For $\varphi, \tilde{\varphi}\in\mathcal{M}^{\alpha},$  put
$$
\|\varphi-\tilde{\varphi}\|_{\mathcal{M}^{\alpha}}=\int_{\mathbb{R}^3}\frac{|\varphi(\xi)-\tilde{\varphi}(\xi)|}{|\xi|^{3+\alpha}}d\xi.
$$
It is proved in Theorem 1.1 of \cite{Morimoto-Wang-Yang} that
$\mathcal{M}^{\alpha}\subset\mathcal{F}(P_{\alpha}(\mathbb{R}^3))\subsetneq\mathcal{K}^{\alpha}$ for $\alpha\in ]0, 2[$ .

\smallskip
\noindent{\bf The proof of \eqref{ex2}.} For any $0<\beta<\alpha<2$, we are ready to prove \eqref{ex2},  which shows the relation between $Q^{-\frac{3}{2}-\beta}_r(\mathbb{R}^3)$ and $\mathcal{M}^{\alpha}$.

Let $g\in Q^{-\frac{3}{2}-\beta}_r(\mathbb{R}^3)$  with $\langle g^0, \sqrt{\mu}\rangle=0$, recalled that $\varphi_0=\sqrt{\mu}$,
we have the decomposition
$$
g=\sum^{+\infty}_{k=1}g_k\, \varphi_{k},\,\,\,\mbox{with}\,\,\, g_k=\langle g,\, \varphi_{k}\rangle\, .
$$
Then
$$
\|\widehat{\mu}+\widehat{\sqrt{\mu}g}\, -1\|_{\mathcal{M}^{\alpha}}
=\int_{\mathbb{R}^3}\frac{|\mu(\xi)+\mathcal{F}(\sum^{+\infty}_{k=1}g_k\sqrt{\mu}\varphi_{k})(\xi)-1|}{|\xi|^{3+\alpha}}d\xi\, .
$$
On the other hand,
$$
1=\widehat{\delta_0}=\widehat{e^{-\frac{|v|^2}{4}}\delta_0}=(2\pi)^{\frac{3}{4}}\widehat{\sqrt{\mu}\, \delta_0},
$$
we have
\begin{align*}
(2\pi)^{\frac{3}{4}}\sqrt{\mu}\, \delta_0&=(2\pi)^{\frac{3}{4}}\sum^{+\infty}_{k=0}\langle \delta_0,\varphi_{k}\rangle\sqrt{\mu}\, \varphi_{k}\\
&=(2\pi)^{\frac{3}{4}}\sum^{+\infty}_{k=0}\frac{1}{\pi}
\sqrt{\frac{\Gamma(k+\frac{3}{2})}{\sqrt{2}k!}}\sqrt{\mu}\, \varphi_{k}\\
&=\mu+(2\pi)^{\frac{3}{4}}\sum^{+\infty}_{k=1}\frac{1}{\pi}
\sqrt{\frac{\Gamma(k+\frac{3}{2})}{\sqrt{2}k!}}\sqrt{\mu}\, \varphi_{k}=\mu+\sum^{+\infty}_{k=1}a_k\sqrt{\mu}\, \varphi_{k},
\end{align*}
with $$a_k=\sqrt{\frac{2\Gamma(k+\frac{3}{2})}{\sqrt{\pi}k!}},$$
one can verify that
$$
\|\widehat{\mu}+\widehat{\sqrt{\mu}g}\, -1\|_{\mathcal{M}^{\alpha}}
=\int_{\mathbb{R}^3}\frac{|\mathcal{F}(\sum^{+\infty}_{k=1}(g_k-a_k)\sqrt{\mu}\, \varphi_{k})(\xi)|}{|\xi|^{3+\alpha}}d\xi.
$$
Recalled Lemma 3.2 in \cite{NYKC3} (see also Lemma 6.6 in \cite{GLX}) that
$$
\mathcal{F}(\sqrt{\mu}\, \varphi_{k})=\frac{1}{\sqrt{(2k+1)!}}|\xi|^{2k}e^{-\frac{|\xi|^2}{2}}.
$$
By Fubini theorem,
\begin{align*}
\|\widehat{\mu}+\widehat{\sqrt{\mu}g}\, -1\|_{\mathcal{M}^{\alpha}}=&\int_{\mathbb{R}^3}\frac{|\sum^{+\infty}_{k=1}(g_k-a_k)\mathcal{F}(\sqrt{\mu}\,\varphi_{k})(\xi)|}{|\xi|^{3+\alpha}}d\xi\\
=&\int_{\mathbb{R}^3}\frac{|\sum^{+\infty}_{k=1}(g_k-a_k)
\frac{1}{\sqrt{(2k+1)!}}|\xi|^{2k}e^{-\frac{|\xi|^2}{2}}|}{|\xi|^{3+\alpha}}d\xi\\
\leq&\sum^{+\infty}_{k=1}
\frac{|g_k-a_k|}{\sqrt{(2k+1)!}}\int_{\mathbb{R}^3}\frac{|\xi|^{2k}e^{-\frac{|\xi|^2}{2}}}{|\xi|^{3+\alpha}}d\xi\\
=&4\pi\sum^{+\infty}_{k=1}
\frac{|g_k-a_k|}{\sqrt{(2k+1)!}}2^{k-1-\frac{\alpha}{2}}\Gamma(k-\frac{\alpha}{2}).
\end{align*}
By using the Stirling equivalent
$$\Gamma(x+1)\sim_{x\rightarrow+\infty}\sqrt{2\pi x}\left(\frac{x}{e}\right)^x,$$
it follows that, $\forall k\geq1$
$$2^{k-1-\frac{\alpha}{2}}\frac{\Gamma(k-\frac{\alpha}{2})}{\sqrt{(2k+1)!}}\sim k^{\frac{5}{4}+\frac{\alpha}{2}},\quad a_k\sim k^{\frac{1}{4}}.$$
Therefore
\begin{align*}
\|\widehat{\mu}+\widehat{\sqrt{\mu}g}-1\|_{\mathcal{M}^{\alpha}}
&\lesssim\sum^{+\infty}_{k=1}\frac{|g_k|+|a_k|}{k^{\frac{5}{4}+\frac{\alpha}{2}}}\\
&\lesssim\sum^{+\infty}_{k=1}\frac{|g_k|}{k^{\frac{5}{4}+\frac{\alpha}{2}}}+\sum^{+\infty}_{k=1}\frac{1}{k^{1+\frac{\alpha}{2}}}.
\end{align*}
Recalled the definition of  $ Q^{-\frac{3}{2}-\beta}_r(\mathbb{R}^3)$  that
$$
\sum^{+\infty}_{k=1}\frac{|g_k|}{k^{\frac{5}{4}+\frac{\alpha}{2}}}\leq \left(\sum^{+\infty}_{k=1}k^{-\frac{3}{2}-\beta}|g_k|^2\right)^{\frac{1}{2}} \left(\sum^{+\infty}_{k=1}k^{-1-\alpha+\beta}\right)^{\frac{1}{2}}\lesssim\|g\|_{Q^{-\frac{3}{2}-\beta}_r(\mathbb{R}^3)}
$$
and
$$
\sum^{+\infty}_{k=1}\frac{1}{k^{1+\frac{\alpha}{2}}}<+\infty,
$$
we conclude that $f=\mu+\sqrt{\mu}g\in\mathcal{F}^{-1}(\mathcal{M}^{\alpha}).$

\smallskip\noindent
{\bf Example.} We show the Gelfand-Shilov smoothing effect with the initial datum \eqref{ex2}, here we only need to prove $g^0\in Q^{-\frac{3}{2}-\alpha}_r(\mathbb{R}^3)\cap \mathcal{N}^{\perp}$.

Recalled the spectrum functions $\varphi_{0}(v)$ and $\varphi_1(v)$ at the beginning of this Section, then we have
$$g^0=\frac{1}{\sqrt{\mu}}\delta_{\mathbf{0}}-\sqrt{\mu}-
\left(\frac{3}{2}-\frac{|v|^2}{2}\right)\sqrt{\mu}
=\frac{1}{\sqrt{\mu}}\delta_{\mathbf{0}}-\varphi_{0}-\sqrt{\frac{3}{2}}\varphi_{1},$$
one can verify that
$$\langle g^0, \varphi_0\rangle=\langle g^0,\varphi_1\rangle=0.$$
This shows that $g^0\in\mathcal{N}^{\perp}$.  Now we prove that $g^0\in \,Q^{-\frac{3}{2}-\alpha}_r(\mathbb{R}^3).$

Since $g^0\in\mathcal{N}^{\perp}$, we can write $g^0$ in the form
$$
g^0=\sum^{+\infty}_{k=2}\langle g^0, \varphi_{k}\rangle\, \varphi_{k},
$$
where we can calculate that, for $k\ge 2$,
$$
\langle g^0, \varphi_{k}\rangle=\langle \mu^{-\frac 12}\delta_{\mathbf{0}}, \varphi_{k}\rangle=\sqrt{\frac{2\Gamma(k+\frac{3}{2})}{\sqrt{\pi}k!}}.
$$
By using the Stirling equivalent
$$
\Gamma(x+1)\sim_{x\rightarrow+\infty}\sqrt{2\pi x}\left(\frac{x}{e}\right)^x,
$$
we have that, $\forall k\geq2$
$$
\sqrt{\frac{2\Gamma(k+\frac{3}{2})}{\sqrt{\pi}k!}}\sim k^{\frac{1}{4}}\, .
$$
Therefore, for any $\alpha>0$,
$$
\|g^0\|^2_{Q^{-\frac{3}{2}-\alpha}_r(\mathbb{R}^3)}
=\sum^{+\infty}_{k=2}k^{-\frac{3}{2}-\alpha}|\langle g^0, \varphi_{k}\rangle|^2
\lesssim\sum^{+\infty}_{k=2}\frac{1}{k^{1+\alpha}}<+\infty.
$$
This implies that $g^0\in\,Q^{-\frac{3}{2}-\alpha}_r(\mathbb{R}^3)$, we end the proof of the Example.


\section{The trilinear estimates for Boltzmann operator}\label{S3}

To prove the convergence of the formal solution obtained in the precedent Section, we need to estimate the following trilinear terms
$$
\left({\bf \Gamma}(f,g),h\right)_{L^2(\mathbb{R}^3)},
\,\,\,f,g,h\in\mathscr{S}_r(\mathbb{R}^3)\cap Q^{-\frac{3}{2}-\alpha}_r(\mathbb{R}^3) \cap \mathcal{N}^{\perp}.
$$
By a proof similar to that in Lemma 3.4 in \cite{NYKC3} or we can refer to Section 6  in \cite{GLX}, we present the estimation properties of the eigenvalues of the linearized radially symmetric Boltzmann operator $\mathcal{L}$, which is a basic tool in the proof of the trilinear estimate with exponential weighted.

\begin{lemma}\label{sum}
For $k\ge2,l\geq2,$
$$\mu_{k,l}=\sqrt{\frac{(2k+2l+1)!}{(2k+1)!(2l+1)!}}
\left(\int^{\frac{\pi}{4}}_0\beta(\theta)(\sin\theta)^{2k}(\cos\theta)^{2l}d\theta\right)$$
was defined in \eqref{mumn}, for $n\geq4$ and for any $0<\alpha<2s<2$ with s given in \eqref{beta}, we have
\begin{equation}\label{estimatesum}
\sum_{\substack{k+l=n\\k\geq2,l\geq2}}\frac{|\mu_{k,l}|^2}{l^{s-\frac{3}{2}-\alpha}k^{-\frac{3}{2}-\alpha}}\lesssim n^{s+\frac{3}{2}+\alpha},
\end{equation}
and
\begin{equation}\label{estimatesum-1}
\sum_{\substack{k+l=n\\k\geq2,l\geq2}}\frac{|\mu_{k,l}|^2}{l^{-\frac{3}{2}-\alpha}k^{-\frac{3}{2}-\alpha}}\lesssim n^{2s+\frac{3}{2}+\alpha}.
\end{equation}

\end{lemma}
\begin{proof}
Since $\beta(\theta)$ satisfies the condition \eqref{beta}, we can begin by noticing that
$$\mu_{k,l}=\sqrt{\frac{(2k+2l+1)!}{(2k+1)!(2l+1)!}}\int^{\frac{\pi}{4}}_0(\sin\theta)^{2k-1-2s}(\cos\theta)^{2l+1}d\theta.$$
By using the substitution rule with $t=(\sin\theta)^2$ that we have
$$\int^{\frac{\pi}{4}}_0(\sin\theta)^{2k-1-2s}(\cos\theta)^{2l+1}d\theta=\frac{1}{2}\int^{\frac{1}{2}}_0t^{k-1-s}(1-t)^{l}dt
\lesssim\frac{\Gamma(k-s)l!}{\Gamma(k+l+1-s)}.$$
Recalled the Stirling equivalent
$$\Gamma(x+1)\sim_{x\rightarrow+\infty}\sqrt{2\pi x}\left(\frac{x}{e}\right)^x,\quad\Gamma(l+1)=l!,$$
for $k\ge2,l\geq2$,  it follows that,
\begin{align*}
\mu_{k,l}&\lesssim\sqrt{\frac{2k+2l+1}{(2k+1)(2l+1)}}\sqrt{\frac{(2k+2l)!}{(2k)!(2l)!}}\frac{\Gamma(k-s)l!}{\Gamma(k+l+1-s)}
\\
&\lesssim\sqrt{\frac{k+l}{kl}}\left(\frac{k+l}{kl}\right)^{\frac{1}{4}}\left(\frac{2k+2l}{e}\right)^{k+l}\left(\frac{e}{2k}\right)^{k}\left(\frac{e}{2l}\right)^{l}\\
&\qquad\times\sqrt{\frac{kl}{k+l}}\left(\frac{k}{e}\right)^{k-s-1}\left(\frac{l}{e}\right)^{l}\left(\frac{e}{k+l}\right)^{k+l-s}\\
&\lesssim\left(\frac{k+l}{kl}\right)^{\frac{1}{4}}\frac{(k+l)^{s}}{k^{1+s}}
=\frac{(k+l)^{\frac{1}{4}+s}}{l^{\frac{1}{4}}k^{\frac{5}{4}+s}}.
\end{align*}
Then for $n\geq4$ and for any $0<\alpha<2s<2$, we have
\begin{align*}
\sum_{\substack{k+l=n\\k\geq2,l\geq2}}\frac{|\mu_{k,l}|^2}{l^{s-\frac{3}{2}-\alpha}k^{-\frac{3}{2}-\alpha}}
&\lesssim\sum_{\substack{k+l=n\\k\geq2,l\geq2}}\frac{(k+l)^{\frac{1}{2}+2s}}{l^{s-1-\alpha}k^{1+2s-\alpha}}\\
&=\sum^{n-2}_{k=2}\frac{n^{2s+\frac{1}{2}}}{(n-k)^{s-1-\alpha}k^{1+2s-\alpha}}\\
&\leq2\sum^{\frac{n}{2}}_{k=2}\frac{n^{s+\frac{3}{2}+\alpha}}{k^{1+2s-\alpha}}+2\sum_{\frac{n}{2}<k\leq n-2}\frac{n^{s+\frac{3}{2}+\alpha}}{(n-k)^{1+2s-\alpha}}\\
&\lesssim n^{s+\frac{3}{2}+\alpha},
\end{align*}
and
\begin{align*}
\sum_{\substack{k+l=n\\k\geq2,l\geq2}}\frac{|\mu_{k,l}|^2}{l^{-\frac{3}{2}-\alpha}k^{-\frac{3}{2}-\alpha}}
&\lesssim\sum_{\substack{k+l=n\\k\geq2,l\geq2}}\frac{(k+l)^{\frac{1}{2}+2s}}{l^{-1-\alpha}k^{1+2s-\alpha}}\\
&=\sum^{n-2}_{k=2}\frac{n^{2s+\frac{1}{2}}}{(n-k)^{-1-\alpha}k^{1+2s-\alpha}}\\
&\leq2\sum^{\frac{n}{2}}_{k=2}\frac{n^{2s+\frac{3}{2}+\alpha}}{k^{1+2s-\alpha}}+2\sum_{\frac{n}{2}<k\leq n-2}\frac{n^{2s+\frac{3}{2}+\alpha}}{(n-k)^{1+2s-\alpha}}\\
&\lesssim n^{2s+\frac{3}{2}+\alpha}.
\end{align*}

These are the formulas \eqref{estimatesum} and \eqref{estimatesum-1}, we end the proof of Lemma \ref{sum}.
\end{proof}

The sharp trilinear estimates for the radially symmetric Boltzmann operator can be derived from the result of Lemma \ref{sum}.
\begin{proposition}\label{estimatetri}
For any $0<\alpha<2s<2$ with $s$ given in \eqref{beta},
there exists a positive constant $C>0$, such that for all $f,g,h\in Q^{-\frac{3}{2}-\alpha}_r(\mathbb{R}^3)\cap \mathcal{N}^{\perp}$ and for any $m\ge 2,\, t\geq0$,
\begin{align*}
&|(\Gamma(\mathbb{S}_m f, \, \mathbb{S}_m g),\mathcal{H}^{-\frac{3}{2}-\alpha}\mathbb{S}_m h)_{L^2}|
\\
\leq& C\|\mathcal{H}^{-\frac{3}{4}-\frac{\alpha}{2}} \mathbb{S}_{m-2}f\|_{L^2}
\|\mathcal{H}^{\frac{s}{2}-\frac{3}{4}-\frac{\alpha}{2}}\mathbb{S}_{m-2} g\|_{L^2}
\|\mathcal{H}^{\frac{s}{2}-\frac{3}{4}-\frac{\alpha}{2}}\mathbb{S}_m h\|_{L^2}
\end{align*}
and for any $m\geq2$, $c>0$,
\begin{align*}
&|(\Gamma(\mathbb{S}_m f, \, \mathbb{S}_m g),e^{ct\mathcal{H}^s}\mathcal{H}^{-\frac{3}{2}-\alpha}\mathbb{S}_m h)_{L^2}|\notag\\
\leq& C\|e^{\frac{1}{2}ct\mathcal{H}^s}\mathcal{H}^{-\frac{3}{4}-\frac{\alpha}{2}}\mathbb{S}_{m-2}f\|_{L^2}
\|e^{\frac{1}{2}ct\mathcal{H}^s}\mathcal{H}^{\frac{s}{2}-\frac{3}{4}-\frac{\alpha}{2}}\mathbb{S}_{m-2}g\|_{L^2}\\
&\qquad \qquad \qquad \qquad \times \|e^{\frac{1}{2}ct\mathcal{H}^s}\mathcal{H}^{\frac{s}{2}-\frac{3}{4}-\frac{\alpha}{2}}\mathbb{S}_m h\|_{L^2}\, .
\end{align*}
where  $\mathbb{S}_m$ is the orthogonal projector onto the $m+1$ energy levels
\begin{equation}\label{SN}
\mathbb{S}_{m}f=\sum^{m}_{k=0}\langle f,\varphi_{k}\rangle\varphi_{k}.
\end{equation}
\end{proposition}
\begin{proof}
Let $f,g,h\in Q^{-\frac{3}{2}-\alpha}_r(\mathbb{R}^3) \cap \mathcal{N}^{\perp}$, by using the spectral decomposition, we obtain
$$
f=\sum^{+\infty}_{n=2}\langle f,\varphi_{n}\rangle\varphi_{n},\quad g=\sum^{+\infty}_{n=2}\langle g,\varphi_{n}\rangle\varphi_{n},\quad h=\sum^{+\infty}_{n=2}\langle h,\varphi_{n}\rangle\varphi_{n}.
$$
In convenience, we rewrite  $f_n=\langle f,\varphi_{n}\rangle, g_n=\langle g,\varphi_{n}\rangle,
 h_n=\langle h,\varphi_{n}\rangle.$
We can deduce formally from Proposition \ref{expansion} that
\begin{align*}
 \Gamma(\mathbb{S}_{m}f,\mathbb{S}_{m}g)
&=\sum^{m}_{k=2}\sum^{m}_{l=2}f_k\, g_l\Gamma(\varphi_{k},\varphi_{l})\\
&=\sum^{m}_{k=2}\sum^{m}_{l=2}f_k\, g_l\mu_{k,l}\varphi_{k+l}.
\end{align*}
By using the orthogonal property of $\{\varphi_{n}; \, n\in\mathbb{N}\}$, we have
\begin{align*}
&(\Gamma(\mathbb{S}_m f, \, \mathbb{S}_m g),\mathcal{H}^{-\frac{3}{2}-\alpha}\mathbb{S}_m h)\\
=&\sum^{m}_{n=4}\left(\sum_{\substack{k+l=n\\k\geq2,l\geq2}}\mu_{k,l}f_kg_l\right)
(2n+\frac{3}{2})^{-\frac{3}{2}-\alpha}h_n.
\end{align*}
Therefore,
\begin{align*}
&|(\Gamma(\mathbb{S}_m f, \, \mathbb{S}_m g),\mathcal{H}^{-\frac{3}{2}-\alpha}\mathbb{S}_m h)_{L^2}|\\
\leq& \sum^{m-2}_{l=2}\sum^{m-l}_{k=2}|\mu_{k,l}||f_k||g_l||(2k+2l+\frac{3}{2})^{-\frac{3}{2}-\alpha}||h_{k+l}|\\
\leq&\left(\sum^{m-2}_{l=2}l^{s-\frac{3}{2}-\alpha}|g_l|^2\right)^{\frac{1}{2}}
\left(\sum^{m-2}_{k=2}|f_k|^2k^{-\frac{3}{2}-\alpha}\right)^{\frac{1}{2}}\\
&\times\left( \sum^{m-2}_{l=2}\sum^{m-l}_{k=2}\frac{|\mu_{k,l}|^2|h_{k+l}|^2(k+l)^{-3-2\alpha}}{l^{s-\frac{3}{2}-\alpha}k^{-\frac{3}{2}-\alpha}}\right)^{\frac{1}{2}}\\
=&\|\mathcal{H}^{-\frac{3}{4}-\frac{\alpha}{2}}\mathbb{S}_m f\|_{L^2}
\|\mathcal{H}^{\frac{s}{2}-\frac{3}{4}-\frac{\alpha}{2}}\mathbb{S}_m g\|_{L^2}\\
&\times\left(\sum^{m}_{n=4}\Big(\sum_{\substack{k+l=n\\k\geq2,l\geq2}}\frac{|\mu_{k,l}|^2}{l^{s-\frac{3}{2}-\alpha}k^{-\frac{3}{2}-\alpha}}\Big)|h_n|^2n^{-3-2\alpha}\right)^{\frac{1}{2}}.
\end{align*}
By using \eqref{estimatesum}, for $n\geq4$, $0<\alpha<2s<2$,
$$
\sum_{\substack{k+l=n\\k\geq2,l\geq2}}\frac{|\mu_{k,l}|^2}{l^{s-\frac{3}{2}-\alpha}k^{-\frac{3}{2}-\alpha}}\lesssim n^{s+\frac{3}{2}+\alpha},
$$
we conclude that
\begin{align*}
&|(\Gamma(\mathbb{S}_m f, \, \mathbb{S}_m g),\mathcal{H}^{-\frac{3}{2}-\alpha}\mathbb{S}_m h)|\\
\leq&C\|\mathcal{H}^{-\frac{3}{4}-\frac{\alpha}{2}}\mathbb{S}_{m-2} f\|_{L^2}
\|\mathcal{H}^{\frac{s}{2}-\frac{3}{4}-\frac{\alpha}{2}}\mathbb{S}_{m-2} g\|_{L^2}
\|\mathcal{H}^{\frac{s}{2}-\frac{3}{4}-\frac{\alpha}{2}}\mathbb{S}_m h\|_{L^2}.
\end{align*}
This is the first result of Proposition \ref{estimatetri}.

For the second one,  we have,
\begin{align*}
&(\Gamma(\mathbb{S}_m f, \, \mathbb{S}_m g),\, e^{ct\mathcal{H}^s}\mathcal{H}^{-\frac{3}{2}-\alpha}\mathbb{S}_m h)\\
=&\sum^{m}_{n=4}e^{ct(2n+\frac{3}{2})^s}(2n+\frac{3}{2})^{-\frac{3}{2}-\alpha}\Big(\sum_{\substack{k+l=n\\k\geq2,l\geq2}}\mu_{k,l}f_kg_l\Big)h_n.
\end{align*}
Then
\begin{align*}
&|(\Gamma(\mathbb{S}_m f, \, \mathbb{S}_m g),\, e^{ct\mathcal{H}^s}\mathcal{H}^{-\frac{3}{2}-\alpha}\mathbb{S}_m h)|\\
\leq&\sum^{m-2}_{l=2}\sum^{m-l}_{k=2}|g_l||\mu_{k,l}||f_k||h_{k+l}|e^{ct(2k+2l+\frac{3}{2})^s}
(2k+2l+\frac{3}{2})^{-\frac{3}{2}-\alpha}\\
\lesssim&\left(\sum^{m-2}_{l=2}e^{c(2l+\frac{3}{2})^st}l^{s-\frac{3}{2}-\alpha}|g_l|^2\right)^{\frac{1}{2}}\\
&\times\left(\sum^{m-2}_{l=2}
\frac{ e^{-c(2l+\frac{3}{2})^st}}{l^{s-\frac{3}{2}-\alpha}}
\Big(\sum^{m-l}_{k=2}e^{c(2k+2l+\frac{3}{2})^st}(k+l)^{-\frac{3}{2}-\alpha}|\mu_{k,l}||f_k||h_{k+l}|\Big)^2\right)^{\frac{1}{2}}\\
\leq&\left(\sum^{m-2}_{l=2}e^{c(2l+\frac{3}{2})^st}l^{s-\frac{3}{2}-\alpha}|g_l|^2\right)^{\frac{1}{2}}
\left(\sum^{m-2}_{k=2}e^{c(2k+\frac{3}{2})^st}k^{-\frac{3}{2}-\alpha}|f_k|^2\right)^{\frac{1}{2}}\\
\times&\left(\sum^{m-2}_{l=2}\frac{1}{l^{s-\frac{3}{2}-\alpha}}
\sum^{m-l}_{k=2}e^{2c(2k+2l+\frac{3}{2})^st-c(2l+\frac{3}{2})^st-c(2k+\frac{3}{2})^st}\frac{(k+l)^{-3-2\alpha}}{k^{-\frac{3}{2}-\alpha}}|\mu_{k,l}|^2|h_{k+l}|^2\right)^{\frac{1}{2}}.
\end{align*}
Using the elementary inequality for $k\ge2, l\ge2$,
$$
(2k+2l+\frac{3}{2})^s\leq(2l+\frac{3}{2})^s+(2k+\frac{3}{2})^s,
$$
one can verify that, for $t\geq0$,
\begin{align*}
&|(\Gamma (\mathbb{S}_m f, \, \mathbb{S}_m g),\, e^{ct\mathcal{H}^s}\mathcal{H}^{-\frac{3}{2}-\alpha}\mathbb{S}_m h)|\\
\leq &\|e^{\frac{1}{2}ct\mathcal{H}^s}\mathcal{H}^{-\frac{3}{4}-\frac{\alpha}{2}}\mathbb{S}_{m-2}f\|_{L^2}
\|e^{\frac{1}{2}ct\mathcal{H}^s}\mathcal{H}^{\frac{s}{2}-\frac{3}{4}-\frac{\alpha}{2}}\mathbb{S}_{m-2}g\|_{L^2}\\
&\qquad\times\left(\sum^{m-2}_{l=2}
\frac{1}{l^{s-\frac{3}{2}-\alpha}}
\sum^{m-l}_{k=2}e^{c(2k+2l+\frac{3}{2})^st}\frac{(k+l)^{-3-2\alpha}}{k^{-\frac{3}{2}-\alpha}}|\mu_{k,l}|^2|h_{k+l}|^2\right)^{\frac{1}{2}}\\
=&\|e^{\frac{1}{2}ct\mathcal{H}^s}\mathcal{H}^{-\frac{3}{4}-\frac{\alpha}{2}}\mathbb{S}_{m-2}f\|_{L^2}
\|e^{\frac{1}{2}ct\mathcal{H}^s}\mathcal{H}^{\frac{s}{2}-\frac{3}{4}-\frac{\alpha}{2}}\mathbb{S}_{m-2}g\|_{L^2}\\
&\qquad\times
\left(\sum^{m}_{n=4}e^{c(2n+\frac{3}{2})^st}n^{-3-2\alpha}
\Big(\sum_{\substack{k+l=n\\k\geq2,l\geq2}}\frac{|\mu_{k,l}|^2}{l^{s-\frac{3}{2}-\alpha}k^{-\frac{3}{2}-\alpha}}\Big)|h_n|^2\right)^{\frac{1}{2}}.
\end{align*}
By using Lemma \ref{sum} again that, for $n\geq4$, $0<\alpha<2s<2$,
$$
\sum_{\substack{k+l=n\\k\geq2,l\geq2}}\frac{|\mu_{k,l}|^2}{l^{s-\frac{3}{2}-\alpha}k^{-\frac{3}{2}-\alpha}}\lesssim n^{s+\frac{3}{2}+\alpha}.
$$
We conclude that, for $t\geq0$, $m\geq2$,
\begin{align*}
&|(\Gamma(\mathbb{S}_m f, \, \mathbb{S}_m g),\, e^{ct\mathcal{H}^s}\mathcal{H}^{-\frac{3}{2}-\alpha}\mathbb{S}_m h)_{L^2}|\\
\lesssim&\|e^{\frac{1}{2}ct\mathcal{H}^s}\mathcal{H}^{-\frac{3}{4}-\frac{\alpha}{2}}\mathbb{S}_{m-2}f\|_{L^2}
\|e^{\frac{1}{2}ct\mathcal{H}^s}\mathcal{H}^{\frac{s}{2}-\frac{3}{4}-\frac{\alpha}{2}}\mathbb{S}_{m-2}g\|_{L^2}\\
&\qquad\qquad\times\left(\sum^{m}_{n=4}e^{c(2n+\frac{3}{2})^st}n^{s-\frac{3}{2}-\alpha}|h_n|^2\right)^{\frac{1}{2}}\\
\leq& C\|e^{\frac{1}{2}ct\mathcal{H}^s}\mathcal{H}^{-\frac{3}{4}-\frac{\alpha}{2}}\mathbb{S}_{m-2}f\|_{L^2}
\|e^{\frac{1}{2}ct\mathcal{H}^s}\mathcal{H}^{\frac{s}{2}-\frac{3}{4}-\frac{\alpha}{2}}\mathbb{S}_{m-2}g\|_{L^2}\\
&\qquad\qquad\qquad\times
\|e^{\frac{1}{2}ct\mathcal{H}^s}\mathcal{H}^{\frac{s}{2}-\frac{3}{4}-\frac{\alpha}{2}}\mathbb{S}_m h\|_{L^2}.
\end{align*}
This ends the proof of Proposition \ref{estimatetri}.
\end{proof}


\section{Estimates of the formal solutions}\label{S4}

In this Section, we study the convergence of the formal solutions obtained on Section \ref{S2} with small initial data in Shubin spaces.

\smallskip\noindent
{\bf The uniform estimate. } Let $\{g_n(t); n\in\mathbb{N}\}$ be the solution of \eqref{ODE-2},  then $  \mathbb{S}_Ng(t) \in\mathscr{S}_r(\mathbb{R}^3)$ for any $N\in \mathbb{N}$.
Multiplying $e^{c_0t(2n+\frac{3}{2})^s}(2n+\frac{3}{2})^{-\frac{3}{2}-\alpha}\overline{g_n}(t)$ on both sides of \eqref{ODE-1} with $c_0>0$ given in \eqref{eq:3.111} and take summation for $ 2\leq n\leq N$, then Proposition \ref{expansion} and the orthogonality of the basis $\{\varphi_{n}\}_{n\in \mathbb{N}}$ imply that
\begin{align*}
&\Big(\partial_t (\mathbb{S}_Ng)(t),e^{c_0t\mathcal{H}^s}\mathcal{H}^{-\frac{3}{2}-\alpha}\mathbb{S}_N g(t)\Big)_{L^2(\mathbb{R}^3)}\\
&\qquad\qquad+\Big(\mathcal{L}(\mathbb{S}_Ng)(t),e^{c_0t\mathcal{H}^s}\mathcal{H}^{-\frac{3}{2}-\alpha}\mathbb{S}_N g(t)\Big)_{L^2(\mathbb{R}^3)}\\
&=\Big(\Gamma((\mathbb{S}_Ng),(\mathbb{S}_Ng)), e^{c_0t\mathcal{H}^s}\mathcal{H}^{-\frac{3}{2}-\alpha}\mathbb{S}_N g(t)\Big)_{L^2(\mathbb{R}^3)}.
\end{align*}
Since
$$
\Big(\mathcal{L}(\mathbb{S}_Ng)(t),e^{c_0t\mathcal{H}^s}\mathcal{H}^{-\frac{3}{2}-\alpha}\mathbb{S}_N g(t)\Big)_{L^2(\mathbb{R}^3)}=
\|e^{\frac{c_0t}{2}\mathcal{H}^s}\mathcal{L}^{\frac 12}\mathcal{H}^{-\frac{3}{4}-\frac{\alpha}{2}}\mathbb{S}_Ng(t)\|^2_{L^2(\mathbb{R}^3)},
$$
we have
\begin{align*}
&\qquad\qquad\frac{1}{2}\frac{d}{dt}\|e^{\frac{1}{2}c_0t\mathcal{H}^s}\mathcal{H}^{-\frac{3}{4}-\frac{\alpha}{2}}\mathbb{S}_Ng(t)\|^2_{L^2}\\
&\quad-\frac{c_0}{2}\|e^{\frac{1}{2}c_0t\mathcal{H}^s}\mathcal{H}^{\frac{s}{2}-\frac{3}{4}-\frac{\alpha}{2}}\mathbb{S}_Ng(t)\|^2_{L^2(\mathbb{R}^3)}
+\|e^{\frac{1}{2}c_0t\mathcal{H}^s}\mathcal{L}^{\frac 12}\mathcal{H}^{-\frac{3}{4}-\frac{\alpha}{2}}\mathbb{S}_Ng(t)\|^2_{L^2(\mathbb{R}^3)}\\
&\qquad=\Big(\Gamma((\mathbb{S}_Ng),(\mathbb{S}_Ng)), e^{c_0t\mathcal{H}^s}\mathcal{H}^{-\frac{3}{2}-\alpha}\mathbb{S}_Ng(t)\Big)_{L^2}.
\end{align*}
It follows from the inequality \eqref{eq:3.111}
in Remark \ref{remark-decomp}, for $n\geq2$,
$$
 c_0 n^s\leq \lambda_{n}\leq\frac{1}{c_0} n^s
$$
and  Proposition \ref{estimatetri} that, for any $N\geq2$, $t\geq0$,
\begin{align}\label{ele}
&\frac{1}{2}\frac{d}{dt}\|e^{\frac{1}{2}c_0t\mathcal{H}^s}\mathcal{H}^{-\frac{3}{4}-\frac{\alpha}{2}}\mathbb{S}_Ng(t)\|^2_{L^2}
+\frac{1}{2}\|e^{\frac{1}{2}c_0t\mathcal{H}^s}\mathcal{L}^{\frac 12}\mathcal{H}^{-\frac{3}{4}-\frac{\alpha}{2}}\mathbb{S}_Ng(t)\|^2_{L^2(\mathbb{R}^3)}\nonumber\\
&\qquad\qquad \leq C\|e^{\frac{1}{2}c_0t\mathcal{H}^s}\mathcal{H}^{-\frac{3}{4}-\frac{\alpha}{2}}\mathbb{S}_{N-2}g\|_{L^2}
\|e^{\frac{1}{2}c_0t\mathcal{H}^s}\mathcal{L}^{\frac 12}\mathcal{H}^{-\frac{3}{4}-\frac{\alpha}{2}}\mathbb{S}_Ng\|^2_{L^2}.
\end{align}

\begin{proposition}\label{induction} Let $0<\alpha<2s<2$.
There exist  $\epsilon_ 0>0$ such that for all $g^0\in Q^{-\frac{3}{2}-\alpha}_r(\mathbb{R}^3) \cap \mathcal{N}^{\perp} $ with
$$
\|g^0\|^2_{Q^{-\frac{3}{2}-\alpha}_r(\mathbb{R}^3)}
=\sum^{+\infty}_{k=2}(2k+\frac{3}{2})^{-\frac{3}{2}-\alpha}g_k^2\leq\epsilon_0\, ,
$$
if $\{g_n(t); n\in\mathbb{N}\}$ is the solution of \eqref{ODE-2}, then, for any $t\geq0,\, N\geq 0$,
\begin{align}
\|e^{\frac{1}{2}c_0t\mathcal{H}^s}\mathcal{H}^{-\frac{3}{4}-\frac{\alpha}{2}}\mathbb{S}_Ng(t)\|^2_{L^2}
+\frac{1}{2}\int^t_0\|e^{\frac{1}{2}c_0\tau\mathcal{H}^s}\mathcal{L}^{\frac{1}{2}}\mathcal{H}^{-\frac{3}{4}-\frac{\alpha}{2}}
\mathbb{S}_{N}g(\tau)\|^2_{L^2}d\tau\label{est-1}\\
\le\|g^0\|^2_{Q^{-\frac{3}{2}-\alpha}(\mathbb{R}^3)},\notag
\end{align}
where $c_0>0$ is given in \eqref{eq:3.111}. We have also, for any $t\geq0,\, N\geq 0$ and $\gamma>1$,
\begin{equation}\label{est-2}
\|\mathbb{S}_N \Gamma	(\mathbb{S}_Ng(t), \mathbb{S}_Ng(t))\|_{Q^{-2s-\frac{3}{2}-\alpha-\gamma}(\mathbb{R}^3)}\le C_\gamma\|g^0\|^2_{Q^{-\frac{3}{2}-\alpha}(\mathbb{R}^3)},
\end{equation}
where the constant $C_\gamma>0$ depends only on $\gamma$.
\end{proposition}

\begin{proof} Set
$$
\mathcal{E}_t(\mathbb{S}_{N}g)
=\|e^{\frac{1}{2}c_0t\mathcal{H}^s}\mathcal{H}^{-\frac{3}{4}-\frac{\alpha}{2}}\mathbb{S}_Ng(t)\|^2_{L^2}
+\frac{1}{2}\int^t_0\|e^{\frac{1}{2}c_0\tau\mathcal{H}^s}\mathcal{L}^{\frac{1}{2}}\mathcal{H}^{-\frac{3}{4}-\frac{\alpha}{2}}
\mathbb{S}_{N}g(\tau)\|^2_{L^2}d\tau\, .
$$
We prove by induction on $N$.

{\bf 1). For $N\le 2$.}  It follows from \eqref{ODE-2} that
$$g_0(t)=g_1(t)=0,\quad\,g_2(t)=e^{-\lambda_{2}t}g_2(0).$$
Then
$$\mathcal{E}_t(\mathbb{S}_{0}g)=\mathcal{E}_t(\mathbb{S}_1g)=0,$$
and
\begin{align*}
&\mathcal{E}_t(\mathbb{S}_2g)\\
&=e^{c_0t(\frac{11}{2})^{s}}\left(\frac{11}{2}\right)^{-\frac{3}{2}-\alpha}|g_2(t)|^2
+\frac{1}{2}\int^t_0e^{c_0\tau(\frac{11}{2})^{s}}\lambda_2\left(\frac{11}{2}\right)^{-\frac{3}{2}-\alpha}|g_2(\tau)|^2d\tau\\
&=\left[e^{(c_0(\frac{11}{2})^{s}-2\lambda_{2})t}
+\frac{\lambda_2}{2}\int^t_0e^{(c_0(\frac{11}{2})^{s}-2\lambda_{2})\tau}d\tau\right]
\left(\frac{11}{2}\right)^{-\frac{3}{2}-\alpha}|g_2(0)|^2\\
&=\frac{\frac{\lambda_2}{2}}{2\lambda_2-c_0(\frac{11}{2})^s}\left(\frac{11}{2}\right)^{-\frac{3}{2}-\alpha}|g_2(0)|^2\\
&\qquad+\left(1-\frac{\frac{\lambda_2}{2}}{2\lambda_2-c_0(\frac{11}{2})^s}\right)
e^{(c_0(\frac{11}{2})^{s}-2\lambda_2)t}\left(\frac{11}{2}\right)^{-\frac{3}{2}-\alpha}|g_2(0)|^2\\
&\leq\left(\frac{11}{2}\right)^{-\frac{3}{2}-\alpha}|g_2(0)|^2\le \|g^0\|^2_{Q^{-\frac{3}{2}-\alpha}_r(\mathbb{R}^3)},
\end{align*}
where the inequality $c_0(\frac{11}{2})^{s}<\lambda_{2}$ is used in the final formula.

{\bf 2). For $N> 2$.} We want to prove that
$$
\mathcal{E}_t(\mathbb{S}_{N-1}g)
\le\|g^0\|^2_{Q^{-\frac{3}{2}-\alpha}_r(\mathbb{R}^3)}\leq\epsilon_0,
$$
imply
$$
\mathcal{E}_t(\mathbb{S}_{N}g)
\le\|g^0\|^2_{Q^{-\frac{3}{2}-\alpha}_r(\mathbb{R}^3)}\leq\epsilon_0.
$$

Taking now $\epsilon_0>0$ such that,
$$
0<\epsilon_0\leq\frac{1}{4 C}
$$
where $ C$ is defined in Proposition \ref{estimatetri}, we notice that
 $$
 \|e^{\frac{1}{2}c_0t\mathcal{H}^s}\mathcal{H}^{-\frac{3}{4}-\frac{\alpha}{2}}\mathbb{S}_{N-2}g\|^2_{L^2}
\leq\|e^{\frac{1}{2}c_0t\mathcal{H}^s}\mathcal{H}^{-\frac{3}{4}-\frac{\alpha}{2}}\mathbb{S}_{N-1}g\|^2_{L^2}\leq
\mathcal{E}_t(\mathbb{S}_{N-1}g)
\leq\epsilon_0.
$$
Then we deduce from \eqref{ele} that
\begin{align*}
&\frac{1}{2}\frac{d}{dt}\|e^{\frac{1}{2}c_0t\mathcal{H}^s}\mathcal{H}^{-\frac{3}{4}-\frac{\alpha}{2}}\mathbb{S}_Ng(t)\|^2_{L^2}
+\frac{1}{2}\|e^{\frac{1}{2}c_0t\mathcal{H}^s}\mathcal{L}^{\frac{1}{2}}
\mathcal{H}^{-\frac{3}{4}-\frac{\alpha}{2}}\mathbb{S}_Ng(t)\|^2_{L^2}\\
&\leq  C\|e^{\frac{1}{2}c_0t\mathcal{H}^s}\mathcal{H}^{-\frac{3}{4}-\frac{\alpha}{2}}\mathbb{S}_{N-2}g\|_{L^2}
\|e^{\frac{1}{2}c_0t\mathcal{H}^s}\mathcal{L}^{\frac{1}{2}}
\mathcal{H}^{-\frac{3}{4}-\frac{\alpha}{2}}\mathbb{S}_Ng\|^2_{L^2}\\
&\leq\frac{1}{4}\|e^{\frac{1}{2}c_0t\mathcal{H}^s}\mathcal{L}^{\frac{1}{2}}
\mathcal{H}^{-\frac{3}{4}-\frac{\alpha}{2}}\mathbb{S}_Ng\|^2_{L^2},
\end{align*}
therefore,
\begin{equation}\label{indu1}
\frac{d}{dt}\|e^{\frac{1}{2}c_0t\mathcal{H}^s}\mathcal{H}^{-\frac{3}{4}-\frac{\alpha}{2}}\mathbb{S}_Ng(t)\|^2_{L^2}
+\frac{1}{2}\|e^{\frac{1}{2}c_0t\mathcal{H}^s}\mathcal{L}^{\frac{1}{2}}
\mathcal{H}^{-\frac{3}{4}-\frac{\alpha}{2}}\mathbb{S}_Ng(t)\|^2_{L^2}\le 0.
\end{equation}
This ends the proof of \eqref{est-1}.

For the estimate \eqref{est-2}, using the Proposition \ref{expansion} and \eqref{estimatesum-1}, we have
\begin{align*}
&\|\mathbb{S}_N \Gamma	(\mathbb{S}_Ng(t), \mathbb{S}_Ng(t))\|^2_{Q^{-2s-\frac{3}{2}-\alpha-\gamma}(\mathbb{R}^3)}\\
=&\sum^{N}_{n=4}(2n+\frac{3}{2})^{-2s-\frac{3}{2}-\alpha-\gamma}\Big|\sum_{\substack{k+l=n\\k\geq2,l\geq2}}\mu_{k,l}g_k(t)g_l(t)\Big|^2\\
\le &\sum^{N}_{n=4}(2n+\frac{3}{2})^{-2s-\frac{3}{2}-\alpha-\gamma}
\sum_{\substack{k+l=n\\k\geq2,l\geq2}}\frac{|\mu_{k,l}|^2}{l^{-\frac{3}{2}-\alpha}k^{-\frac{3}{2}-\alpha}}
\\
&\qquad\qquad\times \sum_{\substack{k+l=n\\k\geq2,l\geq2}}l^{-\frac{3}{2}-\alpha}k^{-\frac{3}{2}-\alpha} |g_k(t)g_l(t)|^2\\
\le &\sum^{N}_{n=4}(2n+\frac{3}{2})^{-\gamma}\|\mathcal{H}^{-\frac{3}{4}-\frac{\alpha}{2}}\mathbb{S}_Ng(t)\|^4_{L^2}.
\end{align*}
Then \eqref{est-1} imply,  for any $t\geq0,\, N\geq 0$,
\begin{equation*}
\|\mathbb{S}_N \Gamma	(\mathbb{S}_Ng(t), \mathbb{S}_Ng(t))\|^2_{Q^{-2s-\frac{3}{2}-\alpha-\gamma}(\mathbb{R}^3)}\le C_\gamma\|g^0\|^4_{Q^{-\frac{3}{2}-\alpha}(\mathbb{R}^3)},
\end{equation*}
with ($\gamma>1$)
$$
C_\gamma= \sum^{+\infty}_{n=4}(2n+\frac{3}{2})^{-\gamma},
$$
which ended the proof of the Proposition \ref{induction}.
\end{proof}

Using the exact the same proof, we get the following surprise results
\begin{corollary}\label{cor1}
Let $0<\alpha<2s<2$,  then for any $f, g \in Q^{-\frac{3}{2}-\alpha}(\mathbb{R}^3)$, we have 
\begin{equation}\label{tri}
\Gamma(f,\, g)\in Q^{-2s-\frac{3}{2}-\alpha-\gamma}(\mathbb{R}^3)	
\end{equation}
for any $\gamma>1$, and 
\begin{equation*}
\|\Gamma(f, \, g)\|_{Q^{-2s-\frac{3}{2}-\alpha-\gamma}(\mathbb{R}^3)}\le C_\gamma\|f\|_{Q^{-\frac{3}{2}-\alpha}(\mathbb{R}^3)}
\|g\|_{Q^{-\frac{3}{2}-\alpha}(\mathbb{R}^3)}\, ,
\end{equation*}
with $C_\gamma>$ a constant depends only on $\gamma$.
\end{corollary}

\smallskip
\noindent
{\bf Convergence in Shubin space.}\,\, We prove now the convergence of the sequence
$$
g(t)=\sum^{+\infty}_{n=2} g_{n}(t)\varphi_{n}
$$
defined in \eqref{ODE}.  For all $N\geq2$,
$\mathbb{S}_{N}g(t)$ satisfies the following Cauchy problem
\begin{equation}\label{eq-2}
\left\{ \begin{aligned}
         &\partial_t \mathbb{S}_{N}g+\mathcal{L}(\mathbb{S}_{N}g)=\mathbb{S}_{N}\Gamma(\mathbb{S}_{N}g, \mathbb{S}_{N}g),\,\\
         &\mathbb{S}_{N}g|_{t=0}=\sum^{N}_{n=2}\langle g^0,\varphi_{n}\rangle\varphi_{n}.
\end{aligned} \right.
\end{equation}
By Proposition \ref{induction} and the orthogonality of the basis $(\varphi_{n})_{n\in\mathbb{N}}$, for all $t>0$,
\begin{align*}
&\sum^{N}_{n=2}e^{c_0t(2n+\frac{3}{2})^s}(2n+\frac{3}{2})^{-\frac{3}{4}-\alpha}|g_n(t)|^2\\
&+\frac{1}{2}\int^t_0\left(\sum^{N}_{n=2}e^{c_0t(2n+\frac{3}{2})^s}\lambda_n(2n+\frac{3}{2})^{-\frac{3}{4}-\alpha}|g_n(\tau)|^2\right)d\tau
\le\|g^0\|^2_{Q^{-\frac{3}{2}-\alpha}_r(\mathbb{R}^3)}.
\end{align*}
It follows that for all $t\geq0$,
\begin{align}\label{Lestimate}
&\sum^{N}_{n=2}(2n+\frac{3}{2})^{-\frac{3}{2}-\alpha}|g_n(t)|^2
\nonumber\\
&+\frac{1}{2}\int^t_0\left(\sum^{N}_{n=2}\lambda_n(2n+\frac{3}{2})^{-\frac{3}{2}-\alpha}|g_n(\tau)|^2\right)d\tau
\le\|g^0\|^2_{Q^{-\frac{3}{2}-\alpha}_r(\mathbb{R}^3)}.
\end{align}
The orthogonality of the basis $(\varphi_{n})_{n\in\mathbb{N}}$ implies that
$$
\|\mathbb{S}_{N}g(t)\|^2_{Q^{-\frac{3}{2}-\alpha}_r(\mathbb{R}^3)}=\|\mathcal{H}^{-\frac{3}{4}-\frac{\alpha}{2}}\mathbb{S}_{N}g(t)\|^2_{L^2(\mathbb{R}^3)}=
\sum^{N}_{n=2}(2n+\frac{3}{2})^{-\frac{3}{2}-\alpha}|g_n(t)|^2$$
$$\|\mathcal{L}^{\frac{1}{2}}\mathcal{H}^{-\frac{3}{4}-\frac{\alpha}{2}}\mathbb{S}_{N}g(t)\|^2_{L^2(\mathbb{R}^3)}=
\sum^{N}_{n=2}\lambda_n(2n+\frac{3}{2})^{-\frac{3}{2}-\alpha}|g_n(t)|^2.
$$
By using the monotone convergence theorem, we have
$$
\mathbb{S}_{N}g\,\, \to\,\, g(t)=\sum^{+\infty}_{n=2}\,\,
g_n(t)\varphi_{n}\,\,\, \mbox{in} \,\,\, Q^{-\frac{3}{2}-\alpha}_r(\mathbb{R}^3).
$$
Moreover, for any $T> 0$,
\begin{equation}\label{convergence1}
\lim_{N\to \infty}\|\mathbb{S}_{N}g-g\|_{L^\infty([0, T]; Q^{-\frac{3}{2}-\alpha}_r(\mathbb{R}^3))}=0
\end{equation}
and
\begin{equation}\label{convergence2}
\lim_{N\to \infty}\|\mathcal{H}^{\frac{s}{2}}(\mathbb{S}_{N}g-g)\|_{L^2([0, T]; Q^{-\frac{3}{2}-\alpha}_r(\mathbb{R}^3))}=0.
\end{equation}
On the other hand, using  \eqref{estimatesum-1} and \eqref{est-2}, we have also
\begin{equation}\label{Gamma}
\mathbb{S}_{N}\Gamma(\mathbb{S}_{N}g, \mathbb{S}_{N}g)
\,\,\to\,\, \Gamma(g, g)=\sum^{+\infty}_{n=4}\Big(\sum_{\substack{k+l=n\\k\geq2,l\geq2}}\mu_{k,l}f_kg_l\Big)\varphi_n,
\end{equation}
in $Q^{-2s-\frac{3}{2}-\alpha-\gamma}_r(\mathbb{R}^3)$.


\section{The proof of the main Theorem}\label{S5}

We recall the definition of weak solution of \eqref{eq-1}:

\begin{definition}
Let $g^0\in \mathcal{S}'(\mathbb{R}^3)$,  $g(t, v)$ is called a weak solution of the Cauchy problem \eqref{eq-1} if it satisfies the following conditions:
\begin{align*}
&g\in C^0([0, +\infty[; \mathcal{S}'(\mathbb{R}^3)), \quad g(0, v)=g^0(v),\\
&\mathcal{L} ( g )\in L^2([0,  T[; \, \mathcal{S}'(\mathbb{R}^3)),\quad
\Gamma (g , g)\in L^2([0,  T[; \mathcal{S}'(\mathbb{R}^3)),\quad\forall T>0,\\
&\langle g(t), \phi(t)\rangle-\langle g^0, \phi(0)\rangle+\int^t_{0}\langle\mathcal{L}g(\tau), \phi(\tau)\rangle d\tau\\
&\qquad\qquad\qquad\qquad=\int^t_{0}\langle\Gamma(g(\tau),g(\tau)), \phi(\tau)\rangle d\tau,\quad \forall t\ge 0,	
\end{align*}
For any $\phi(t)\in\,C^1\big([0, +\infty[; \mathscr{S}(\mathbb{R}^3)\big)$.
\end{definition}

We prove now the main Theorem \ref{trick}.

\smallskip\noindent
{\bf Existence of weak solution.}

Let $\{ g_n, n\ge 2\}$ be the solutions of the infinite system \eqref{ODE-2} with the initial datum give in the Proposition \ref{induction}, then for any $N\ge 2$, $\mathbb{S}_N g$ satisfy the equation \eqref{eq-2}.

We have, firstly, from the Proposition \ref{induction},  for any $N\ge 2, \gamma>1$, there exists positive constant $C>0$,
\begin{align*}
&\|\mathbb{S}_{N}g\|_{L^\infty[0, +\infty[;  Q^{-\frac{3}{2}-\alpha}_r(\mathbb{R}^3))}
\le \|g^0\|_{Q^{-\frac{3}{2}-\alpha}_r(\mathbb{R}^3)}\, ,\\
&\|\mathbb{S}_{N}\mathcal{L} ( g )\|_{L^\infty[0, +\infty[;  Q^{-2s-\frac{3}{2}-\alpha}_r(\mathbb{R}^3))}
\le C\|g^0\|_{Q^{-\frac{3}{2}-\alpha}_r(\mathbb{R}^3)}\, ,\\
&\|\mathbb{S}_{N}\Gamma(\mathbb{S}_{N}g, \mathbb{S}_{N}g)\|_{L^\infty[0, +\infty[;  Q^{-2s-\frac{3}{2}-\alpha-\gamma}_r(\mathbb{R}^3))}
\le C\|g^0\|^2_{Q^{-\frac{3}{2}-\alpha}_r(\mathbb{R}^3)}.	
\end{align*}
So that  the equation \eqref{eq-2} imply that the sequence $\{\frac{d}{dt} \mathbb{S}_N g(t)\}$ is uniformly bounded  in $  Q^{-2s-\frac{3}{2}-\alpha-\gamma}_r(\mathbb{R}^3)$ with respect to $N\in\mathbb{N}$ and $t\in [0, +\infty[$, so the Arzela-Ascoli Theorem imply that
$$
 \mathbb{S}_N g \,\, \to\,\, g\in C^0([0, +\infty[;   Q^{-2s-\frac{3}{2}-\alpha-\gamma}_r(\mathbb{R}^3))
 \subset C^0([0, +\infty[; \mathcal{S}'(\mathbb{R}^3)),
$$
and
$$
g(0)=g^0.
$$
Secondly,  for any $\phi(t)\in\,C^1\Big(\mathbb{R}_+,\mathscr{S}(\mathbb{R}^3)\Big)$,  we have, for any $t>0$,
\begin{align*}
&\Big|\int^t_0\langle\mathbb{S}_{N}\Gamma(\mathbb{S}_{N}g, \mathbb{S}_{N}g)
-\Gamma(g, g),
\phi(\tau)\rangle\,d\tau\Big|\\
&\le\Big|\int^t_0\langle\Gamma(\mathbb{S}_{N}g, \mathbb{S}_{N}g),\mathbb{S}_{N}\phi(\tau)-\phi(\tau)\rangle\,d\tau\Big|\\
&\qquad+\Big|\int^t_0\langle\Gamma(\mathbb{S}_{N}g-g, \mathbb{S}_{N}g),\phi(\tau)\rangle\,d\tau\Big|+\Big|\int^t_0\langle\Gamma(g, \mathbb{S}_{N}g-g),\phi(\tau)\rangle\,d\tau\Big|.
\end{align*}
By using the estimate \eqref{Lestimate}, the first inequality of Proposition \ref{estimatetri} and the inequality \eqref{eq:3.111} in Remark \ref{remark-decomp}, one can verify that,
\begin{align*}
&\Big|\int^t_0\langle\mathbb{S}_{N}\Gamma(\mathbb{S}_{N}g, \mathbb{S}_{N}g)-\Gamma(g, g),
\phi(\tau)\rangle\,d\tau\Big|\\
&\leq\,C\int^t_0\|\mathbb{S}_{N}g\|_{Q^{-\frac{3}{2}-\alpha}(\mathbb{R}^3)}\|\mathcal{H}^{\frac s2}\mathbb{S}_{N}g\|_{Q^{-\frac{3}{2}-\alpha}(\mathbb{R}^3)} \Big\|\mathcal{H}^{\frac s2}(\mathbb{S}_{N}\phi-\phi)\Big\|_{Q^{\frac{3}{2}+\alpha}(\mathbb{R}^3)} dt\\
&\quad+C\int^t_0 \|\mathbb{S}_{N}g-g\|_{Q^{-\frac{3}{2}-\alpha}(\mathbb{R}^3)}\|\mathcal{H}^{\frac s2}\mathbb{S}_{N}g\|_{Q^{-\frac{3}{2}-\alpha}(\mathbb{R}^3)} \|\mathcal{H}^{\frac s2}\phi\|_{Q^{\frac{3}{2}+\alpha}(\mathbb{R}^3)}
dt\\
&\quad+C\int^t_0\|g\|_{Q^{-\frac{3}{2}-\alpha}(\mathbb{R}^3)}
\|\mathcal{H}^{\frac s2}(\mathbb{S}_{N}g-g)\|_{Q^{-\frac{3}{2}-\alpha}(\mathbb{R}^3)}
\|\mathcal{H}^{\frac s2}\phi\|_{Q^{\frac{3}{2}+\alpha}(\mathbb{R}^3)}
dt\\
&\leq\,C\|g_0\|^2_{Q^{-\frac{3}{2}-\alpha}(\mathbb{R}^3)} \Big\|\mathcal{H}^{\frac s2}(
\mathbb{S}_{N}\phi-\phi)\Big\|_{L^2(]0, t[; Q^{\frac{3}{2}+\alpha}(\mathbb{R}^3))} \\
&\quad+C\|\mathbb{S}_{N}g-g\|_{L^\infty([0, t]; Q^{-\frac{3}{2}-\alpha}(\mathbb{R}^3))}\|g_0\|_{Q^{-\frac{3}{2}-\alpha}(\mathbb{R}^3)} \|\mathcal{H}^{\frac s2}\phi\|_{L^2(]0, t[; Q^{\frac{3}{2}+\alpha}(\mathbb{R}^3))}\\
&\quad+C\|g_0\|_{Q^{-\frac{3}{2}-\alpha}(\mathbb{R}^3)}\|\mathcal{H}^{\frac s2}(\mathbb{S}_{N}g-g)\|_{L^2([0, t]; Q^{-\frac{3}{2}-\alpha}(\mathbb{R}^3))}\|\mathcal{H}^{\frac s2}\phi\|_{L^2(]0, t[; Q^{\frac{3}{2}+\alpha}(\mathbb{R}^3))}.
\end{align*}
Then we can deduce from \eqref{convergence1} and \eqref{convergence2} that, $N\rightarrow+\infty$,
$$\int^t_0\langle\mathbb{S}_{N}\Gamma(\mathbb{S}_{N}g, \mathbb{S}_{N}g)-\Gamma(g, g),
\phi(\tau)\rangle \,d\tau\rightarrow0.$$
For any $\phi(t)\in\,C^1\Big(\mathbb{R}_+,\mathscr{S}(\mathbb{R}^3)\Big)$, the Cauchy problem \eqref{eq-2} can be rewrite as follows
\begin{align*}
&\langle \mathbb{S}_{N}g(t), \phi(t)\rangle-\langle \mathbb{S}_{N}g(0), \phi(0)\rangle\\
&=-\int^t_{0}\langle\mathcal{L}\mathbb{S}_{N}g(\tau), \phi(\tau)\rangle d\tau+\int^t_{0}\langle\mathbb{S}_{N}\Gamma(\mathbb{S}_{N}g(\tau),\mathbb{S}_{N}g(\tau)), \phi(\tau)\rangle d\tau
\end{align*}
Let $N\rightarrow+\infty$, we conclude that,
\begin{align*}
&\langle g(t), \phi(t)\rangle-\langle g^0, \phi(0)\rangle\\
&=-\int^t_{0}\langle\mathcal{L}g(\tau), \phi(\tau)\rangle d\tau+\int^t_{0}\langle\Gamma(g(\tau),g(\tau)), \phi(\tau)\rangle d\tau,
\end{align*}
which shows $g\in L^\infty([0, +\infty[; Q^{-\frac{3}{2}-\alpha}(\mathbb{R}^3))$ is a global weak solution of Cauchy problem \eqref{eq-1}.

\smallskip\noindent
{\bf Regularity of the solution.}\,\, For $\mathbb{S}_{N}g$ defined in \eqref{SN}, where $N\ge2$, since
$$
\lambda_{n}\geq\lambda_{2}>0, \, \forall\,n\geq2,
$$
we deduce from the formulas  \eqref{indu1} and the orthogonality of the basis $(\varphi_n)_{n\in\mathbb{N}}$ that
\begin{align*}
&\frac{d}{dt}\|e^{\frac{1}{2}c_0t\mathcal{H}^s}\mathcal{H}^{-\frac{3}{4}-\frac{\alpha}{2}}\mathbb{S}_Ng(t)\|^2_{L^2}
+\frac{\lambda_2}{2}\|e^{\frac{1}{2}c_0t\mathcal{H}^s}\mathcal{H}^{-\frac{3}{4}-\frac{\alpha}{2}}\mathbb{S}_Ng(t)\|^2_{L^2}\\
&\leq\frac{d}{dt}\|e^{\frac{1}{2}c_0t\mathcal{H}^s}\mathcal{H}^{-\frac{3}{4}-\frac{\alpha}{2}}\mathbb{S}_Ng(t)\|^2_{L^2}
+\frac{1}{2}\sum^{N}_{n=2}e^{c_0t(2n+\frac{3}{2})^{s}}\lambda_n(2n+\frac{3}{2})^{-\frac{3}{2}-\alpha}|g_n(t)|^2\\
&= \frac{d}{dt}\|e^{\frac{1}{2}c_0t\mathcal{H}^s}\mathcal{H}^{-\frac{3}{4}-\frac{\alpha}{2}}\mathbb{S}_Ng(t)\|^2_{L^2}
+\frac{1}{2}\|e^{\frac{1}{2}c_0t\mathcal{H}^s}\mathcal{L}^{\frac{1}{2}}\mathcal{H}^{-\frac{3}{4}-\frac{\alpha}{2}}\mathbb{S}_Ng(t)\|^2_{L^2(\mathbb{R}^3)}\le 0.
\end{align*}
We have then
$$
\frac{d}{dt}\Big(e^{\frac{\lambda_2t}{2}}\|e^{\frac{1}{2}c_0t\mathcal{H}^s}\mathcal{H}^{-\frac{3}{4}-\frac{\alpha}{2}}\mathbb{S}_Ng(t)\|^2_{L^2}\Big)
\leq 0,
$$
it deduces that for any $t>0$, and $N\in\mathbb{N}$,
\begin{equation*}
\|e^{\frac{1}{2}c_0t\mathcal{H}^s}\mathcal{H}^{-\frac{3}{4}-\frac{\alpha}{2}}\mathbb{S}_Ng(t)\|_{L^2(\mathbb{R}^3)}
\le\,e^{-\frac{\lambda_2t}{4}}\|\mathcal{H}^{-\frac{3}{4}-\frac{\alpha}{2}}g_0\|_{L^2(\mathbb{R}^3)}
=e^{-\frac{\lambda_2t}{4}}\|g_0\|_{Q^{-\frac{3}{2}-\alpha}}.
\end{equation*}
By using the monotone convergence theorem, we conclude that, there exists a constant $c_0>0$, such that
\begin{equation*}
\|e^{\frac{1}{2}c_0t\mathcal{H}^s}\mathcal{H}^{-\frac{3}{4}-\frac{\alpha}{2}}g(t)\|_{L^2(\mathbb{R}^3)}
\le e^{-\frac{\lambda_2t}{4}}\|g_0\|_{Q^{-\frac{3}{2}-\alpha}(\mathbb{R}^3)}.
\end{equation*}
The proof of Theorem \ref{trick} is completed.



\section{Appendix}\label{appendix}

The important known results but really needed for this paper are presented in this section. For the self-content of paper,
we will present some proof of those properties.

\smallskip\noindent
{\bf Gelfand-Shilov spaces.}  The symmetric Gelfand-Shilov space $S^{\nu}_{\nu}(\mathbb{R}^3)$ can be characterized through the decomposition
into the Hermite basis $\{H_{\alpha}\}_{\alpha\in\mathbb{N}^3}$ and the harmonic oscillator $\mathcal{H}=-\triangle +\frac{|v|^2}{4}$.
{\color{black} For} more details, see Theorem 2.1 in the book \cite{GPR}
\begin{align*}
f\in S^{\nu}_{\nu}(\mathbb{R}^3)&\Leftrightarrow\,f\in C^\infty (\mathbb{R}^3),\exists\, \tau>0, \|e^{\tau\mathcal{H}^{\frac{1}{2\nu}}}f\|_{L^2}<+\infty;\\
&\Leftrightarrow\, f\in\,L^2(\mathbb{R}^3),\exists\, \epsilon_0>0,\,\,\Big\|\Big(e^{\epsilon_0|\alpha|^{\frac{1}{2\nu}}}(f,\,H_{\alpha})_{L^2}\Big)_{\alpha\in\mathbb{N}^3}\Big\|_{l^2}<+\infty;\\
&\Leftrightarrow\,\exists\,C>0,\,A>0,\,\,\|(-\triangle +\frac{|v|^2}{4})^{\frac{k}{2}}f\|_{L^2(\mathbb{R}^3)}\leq AC^k(k!)^{\nu},\,\,\,k\in\mathbb{N}
\end{align*}
where
$$H_{\alpha}(v)=H_{\alpha_1}(v_1)H_{\alpha_2}(v_2)H_{\alpha_3}(v_3),\,\,\alpha\in\mathbb{N}^3,$$
and for $x\in\mathbb{R}$,
$$H_{n}(x)=\frac{(-1)^n}{\sqrt{2^nn!\pi}}e^{\frac{x^2}{2}}\frac{d^n}{dx^n}(e^{-x^2})
=\frac{1}{\sqrt{2^nn!\pi}}\Big(x-\frac{d}{dx}\Big)^n(e^{-\frac{x^2}{2}}).$$
For the harmonic oscillator $\mathcal{H}=-\triangle +\frac{|v|^2}{4}$ of 3-dimension and $s>0$, we have
$$
\mathcal{H}^{\frac{k}{2}} H_{\alpha} = (\lambda_{\alpha})^{\frac{k}{2}}H_{\alpha},\,\, \lambda_{\alpha}=\sum^3_{j=1}(\alpha_j+\frac{1}{2}),\,\,k\in\mathbb{N},\,\alpha\in\mathbb{N}^3.
$$

\smallskip\noindent
{\bf Shubin spaces.} We refer the reader to the works \cite{GPR, Shubin} for the Shubin spaces.
Let $\tau\in\mathbb{R}$, The Shubin spaces $Q^{\tau}(\mathbb{R}^3)$
can be also characterized through the decomposition into the Hermite basis :
\begin{align*}
f\in Q^{\tau}(\mathbb{R}^3)
&\Leftrightarrow\,f\in\,\mathcal{S}'(\mathbb{R}^3),\,\,
\Bigl\|\Bigl(\mathcal{H}+1\Bigr)^{\frac{\tau}{2}} \, f\Bigr\|_{L^2}<+\infty;
\\
&\Leftrightarrow\, f\in\,\mathcal{S}'(\mathbb{R}^3),\,\,
\Big\|\Big(
(|\alpha|+\frac{5}{2})^{\tau/2} (f,\,H_{\alpha})_{L^2}
\Big)_{\alpha\in\mathbb{N}^3}\Big\|_{l^2}<+\infty,
\end{align*}
where $|\alpha|=\alpha_1+\alpha_2+\alpha_3$,
$$H_{\alpha}(v)=H_{\alpha_1}(v_1)H_{\alpha_2}(v_2)H_{\alpha_3}(v_3),\,\,\alpha\in\mathbb{N}^3,$$
and for $x\in\mathbb{R}$, $n\in \mathbb{N}$,
$$
H_{n}(x)=\frac{1}{(2\pi)^{\frac{1}{4}}}\frac{1}{\sqrt{n!}}\Big(\frac{x}{2}-\frac{d}{dx}\Big)^n(e^{-\frac{x^2}{4}}).
$$
Thus, we have
\begin{equation}\label{dual}
Q^{-\tau}(\mathbb{R}^3)=\left(Q^{\tau}(\mathbb{R}^3)\right)'	
\end{equation}
where $\left(Q^{\tau}(\mathbb{R}^3)\right)'$ is the dual space of $Q^{\tau}(\mathbb{R}^3)$.

The following proof is based on the Appendix in \cite{MPX}.
\begin{proof}
Setting $$A_{\pm,j}=\frac{v_j}{2}\mp\frac{d}{dv_j},\quad j=1,2,3,$$
we have, for $\alpha\in\mathbb{N}^3, v\in\mathbb{R}^3,$
$$H_{\alpha}(v)=\frac{1}{\sqrt{\alpha_1!\alpha_2!\alpha_3!}}A^{\alpha_1}_{+,1}H_0(v_1)A^{\alpha_2}_{+,2}H_0(v_2)A^{\alpha_3}_{+,3}H_0(v_3),$$
and for j=1,2,3,
$$A_{+,j}H_{\alpha}=\sqrt{\alpha_j+1}H_{\alpha+e_j},\quad A_{-,j}H_{\alpha}=\sqrt{\alpha_j}H_{\alpha-e_j}(=0\,\text{if}\,\alpha_j=0)$$
where $(e_1,e_2,e_3)$ stands for the canonical basis of $\mathbb{R}^3$.
For the harmonic oscillator $\mathcal{H}=-\triangle +\frac{|v|^2}{4}$ of 3-dimension and $s>0$, one can verify that,
$$\mathcal{H}=\frac{1}{2}\sum^3_{j=1}(A_{+,j}A_{-,j}+A_{-,j}A_{+,j}).$$
Therefore, we have
\begin{align*}
\mathcal{H}H_{\alpha}
&=\frac{1}{2}\sum^3_{j=1}(A_{+,j}A_{-,j}+A_{-,j}A_{+,j})H_{\alpha}\\
&=\frac{1}{2}[\sum^3_{j=1}\sqrt{\alpha_j}A_{+,j}H_{\alpha-e_j}+\sum^3_{j=1}\sqrt{\alpha_j+1}A_{-,j}H_{\alpha+e_j}]\\
&=\frac{1}{2}\sum^3_{j=1}(2\alpha_j+1)H_{\alpha}=\sum^3_{j=1}(\alpha_j+\frac{1}{2})H_{\alpha}.
\end{align*}
By using this spectral decomposition, we conclude that
$$
(\mathcal{H}+1)^{\frac{\tau}{2}} H_{\alpha} = (\lambda_{\alpha}+1)^{\frac{\tau}{2}}H_{\alpha},\,\, \lambda_{\alpha}=\sum^3_{j=1}(\alpha_j+\frac{1}{2}),\,\alpha\in\mathbb{N}^3.
$$
This ends the proof of the another definition of the Shubin space.
\end{proof}

Remark that for $\tau>0$,
$$
Q^{\tau}(\mathbb{R}^3)\subsetneq H^{\tau}(\mathbb{R}^3)
$$
where $H^{\tau}(\mathbb{R}^3)$ is the usuel Sobolev space. In fact,
$$
\mathcal{H} f\in L^2(\mathbb{R}^3)\,\,\Rightarrow\,\, \triangle f,\,  |v|^2 f\,  \in L^2(\mathbb{R}^3).
$$
So that for the negative index, we have by the duality \eqref{dual},
$$
H^{-\tau}(\mathbb{R}^3) \subsetneq Q^{-\tau}(\mathbb{R}^3).
$$

\bigskip
\noindent {\bf Acknowledgements.}
The first author is supported by the Natural Science Foundation of China under Grant No.11626235.

\end{document}